\UseRawInputEncoding
\documentclass[11pt,a4paper]{article}

\usepackage{graphicx}
\usepackage{indentfirst}
\usepackage{float}
\usepackage{mathrsfs}
\usepackage{fancyhdr}
\usepackage{amsfonts,amssymb,amsmath,amsthm}
\usepackage{hyperref}

\newtheorem{thm}{Theorem}[section]
\newtheorem{lemma}[thm]{Lemma}
\newtheorem{cor}[thm]{Corollary}
\newtheorem{prop}[thm]{Proposition}
\newtheorem{example}[thm]{Example}
\newtheoremstyle{rem}{10pt}{10pt}{\rmfamily}{}{\bfseries}{.}{.5em}{}
\theoremstyle{rem}
\newtheorem{rem}[thm]{Remark}
\newtheorem{definition}[thm]{Definition}

\title{\bf Complex valued semi-linear heat equations in super-critical  spaces $E^s_\sigma$}
\author{Jie Chen, \  Baoxiang Wang\footnote{Corresponding author.}, \ Zimeng Wang}
\date{}

\begin{document}

\maketitle
\begin{abstract}
		We consider the Cauchy problem for the complex valued semi-linear heat equation
$$
 \partial_t u - \Delta u -  u^m =0, \  \ u (0,x) = u_0(x),
$$
where $m\geq 2$ is an integer and the initial data  belong to super-critical  spaces $E^s_\sigma$  for which the norms are defined by
$$
\|f\|_{E^s_\sigma} = \|\langle \xi\rangle^\sigma 2^{s|\xi|}\widehat{f}(\xi)\|_{L^2}, \ \  \sigma \in \mathbb{R},  \ s<0.
$$
If $s<0$, then any Sobolev space  $H^{r}$  is a subspace of $E^s_\sigma$, i.e., $\cup_{r \in \mathbb{R}} H^r \subset E^s_\sigma$. We obtain the global existence and uniqueness of the solutions if the initial data belong  to $E^s_\sigma$ ($s<0, \ \sigma \geq d/2-2/(m-1)$) and their Fourier transforms are supported in the first octant,  the smallness conditions on the initial data in $E^s_\sigma$ are not required for the global solutions. Moreover, we show that the error between the solution $u$ and the iteration solution  $u^{(j)}$ is $C^j/(j\,!)^2$. Similar results also hold if the nonlinearity $u^m$ is replaced by an exponential function $e^u-1$.\\

		
	\end{abstract}

\section{Introduction}
Let  $ \mathscr{S}$ be the Schwartz space  and $ \mathscr{S}'$ be its dual space. We write
$$
p_\lambda(f) = \sup_{x\in \mathbb{R}^d}e^{\lambda |x|}|f(x)|, \quad q_{\lambda}(f)= \sup_{\xi\in \mathbb{R}^d} e^{\lambda|\xi|}|\widehat{f}(\xi)|,
$$
\begin{equation*}
		\mathscr{S}_1:=\{f\in \mathscr{S}: p_\lambda(f)+q_\lambda(f)<\infty, ~\forall~\lambda>0 \}.
\end{equation*}	
	$ \mathscr{S}_1$ equipped with the system of semi-norms $\{p_\lambda+q_\lambda\}_{\lambda>0}$ is a complete locally convex linear topological space, which is said to be the Gelfand-Shilov space, cf. \cite{GeSh1968}. We denote by $ \mathscr{S}_1'$ the dual space of $ \mathscr{S}_1$.  One easily sees that
$$ \mathscr{S}_1 \subset  \mathscr{S}, \ \  \mathscr{S}' \subset  \mathscr{S}'_1.$$
$\mathscr{S}_1$ contains the translations and modulations of Gaussian $e^{-\mathrm{i}m x} e^{- |x-n|^2/2}$ and their linear combinations, which are dense in any Sobolev spaces $H^\sigma$,  cf. \cite{GeSh1968,Gr2001}.
The Fourier transforms on $ \mathscr{S}_1'$ can be defined by duality (cf. \cite{FeGrLiWa2021}), namely, for any $f\in  \mathscr{S}'_1$, its Fourier transform $\mathscr{F} f = \widehat{f}$ satisfies
$$
\langle \mathscr{F} f, \,  \varphi \rangle = \langle f, \,  \mathscr{F} \varphi \rangle, \ \ \forall \ \varphi  \in \mathscr{S}_1.
$$

\begin{definition}
Let $s, \sigma \in \mathbb{R}$.   Denote
\begin{equation*}
E^s_\sigma = E^s_\sigma (\mathbb{R}^d) :=\{f\in  \mathscr{S}_1': 	 \langle \xi\rangle^\sigma 2^{s|\xi|}\widehat{f}(\xi)\in L^2(\mathbb{R}^d)\}
\end{equation*}
for which the norm is given by
\begin{align} \label{spaceEs}
\|f\|_{E^{s}_\sigma } = 	\left\|\langle \xi\rangle^\sigma 2^{s|\xi|}\widehat{f}(\xi) \right\|_{L^2 (\mathbb{R}^d)}.
\end{align}
\end{definition}
$E^s_\sigma$ is a Banach space. Let us denote $E^s:= E^s_0$. In the case $s>0$, $E^s$ as an infinitely smooth function space, was introduced in Bj\"orck \cite{Bj1966}.  $E^s_\sigma $ can be regarded as a generalization of Sobolev spaces $H^\sigma:=(I-\Delta)^{-\sigma/2} L^2(\mathbb{R}^d)$ for which the norm is defined by
\begin{align} \label{spaceHs}
  \|f\|_{H^\sigma} := \|\langle \xi\rangle^\sigma \widehat{f} \|_{L^2 (\mathbb{R}^d)}.
\end{align}
At first glance, we find that the norm on $E^s_\sigma$ can be obtained  by appending an exponential regularity weight $2^{s|\xi|}$ in the Sobolev norm \eqref{spaceHs} and $H^\sigma= E^0_\sigma$. However, $E^s_\sigma$ is a rather rough function space. In fact, one has that

\begin{prop} \label{embedding}
Let $s<0$ and $\sigma, r \in \mathbb{R}$. Then we have $H^{r} \subset E^s_\sigma$.
\end{prop}
\begin{proof}
 Since $2^{s|\xi|} \lesssim \langle \xi\rangle^{r-\sigma}$ for $s<0$ and $\sigma, r\in \mathbb{R}$, one easily sees that  $H^{r} \subset E^s_\sigma$.
 \end{proof}

By Proposition \ref{embedding}, we see that $\cup_{r\in \mathbb{R}} H^r$ is a subset of $E^s_\sigma$ if $s<0$. It is easy to see that  $E^s \subset E^s_\sigma$ for $\sigma \leq 0$, and $E^s_\sigma \subset E^s$ for $\sigma \geq 0$.

 The aim of this paper is to study the Cauchy problem for a class of semi-linear heat equations with initial data in $E^s_\sigma$  with $s \leq 0$.
Generally speaking, if a nonlinear evolution equation has a critical Sobolev space $H^{\sigma_c}$, which means that it is well-posed in $H^{\sigma}$ for $\sigma>\sigma_c$  and ill posed in $H^{\sigma}$ for $ \sigma<\sigma_c$, then  $H^\sigma$ with $\sigma>\sigma_c$ ($\sigma<\sigma_c$) is said to be the subcritical (supercritical) Sobolev space. Any Banach function space $X$ satisfying $H^\sigma\subset X$ for some $\sigma<\sigma_c$   is said to be a {\it supercritical space}.
For a nonlinear evolution equation, if it has a critical Sobolev space $H^{s_c}$, then Proposition \ref{embedding} implies that any supercritical Sobolev space $H^r$ ($r<\sigma_c$) is a subspace of $E^s_\sigma$ if $\sigma<0$. So, $E^s_\sigma$ with $s<0$ is also a supercritical space.

Recently, Navier-Stokes equation in supercritical modulation spaces $E^s_{p,1}$ was studied in \cite{FeGrLiWa2021}, where a general idea to handle a class of supercritical initial data was presented. In this paper, we will consider the semi-linear heat equation (SLH)
\begin{equation}\label{model}
		 \partial_t u - \Delta u -   f(u) =0, \  \ u (0,x) = u_0(x)
\end{equation}
in the supercritical spaces $E^s_\sigma$ with $s<0$,  where $f(u)=u^{m}$ (or $f(u)=e^u-1$), $u $ is a complex-valued distribution function of $ (t,x)\in \mathbb{R}_+\times \mathbb{R}^d$, $m\geq 2$ is a natural number.

For the real valued SLH, Fujita \cite{Fu1966} showed  that there does not exist a global solution for any
nontrivial nonnegative and suitable smooth initial data in the case $  m<1+2/d$. Hayakawa \cite{Ha1973}, Kobayashi, Sino and Tanaka \cite{KoSiTa1977}, Weissler \cite{We1981} considered the case $m=1+2/d$ and showed that \eqref{model} has no nontrivial global nonnegative solutions. In the case $m>1+2/d$ and the initial data are suitable small in $L^p$ with $ p = d(m-1)/2 >1$, then \eqref{model} has a unique global solution in $L^p$, cf. Weissler \cite{We1981} and Giga \cite{Gi1986}, and their arguments are also adapted to the complex-valued cases. So, $m=1+2/d$ is a critical power for the global solutions of \eqref{model} with non-negative initial data. The blowup behavior of the solutions of \eqref{model} were studied in \cite{GiUm2006,LeNi1992,MiSo2019,Se2018,Se2020} (see also \cite{QuSo2007}). It seems that the sign-change solutions are more complicated. Mizoguchi and  Yanagida \cite{MiYa1997} found that  $m=1+2/(k+1)$  is the critical index in 1D if $u_0\in H^1_\rho(\mathbb{R})$ changes signs $k$ times, where $ H^1_\rho $ denote the weighted $H^1$ spaces with the weight $\rho= e^{x^2/4}$. In the cases $1<m<1+2/d$, Cazenave, Dickstein and Weissler \cite{CaDiWe2009} constructed some initial value $\psi \in C_0 (\mathbb{R}^d) \cap L^1(\mathbb{R}^d)$  such that \eqref{model} has a unique global solution;  but for initial data $u_0 = c \psi$, the unique solution of \eqref{model} blows up in finite time if $c>0$ is either sufficiently small or sufficiently large.

For the complex-valued SLH, it has a strong relation with the viscous Constantin--Lax--Majda equation \cite{CoLaMa1985,Sch1986},
\begin{align}
v_t =v_{xx} + v Hv, \label{CLMeq}
\end{align}
which is a one dimensional model for the vorticity equation.
Here $v$ is a real valued function and $H$ is the Hilbert transform defined by
$$
[Hf](x) = \frac{1}{\pi}\, p.v.\, \int_{\mathbb{R}}\frac{f(y)}{x-y} dy.
$$
It is easy to see that  $Hv$ satisfies
$$
(Hv)_t = (Hv)_{xx} +\frac{1}{2}
((Hv)^2  - v^2).
$$
Denote $u=\frac{1}{2 {\rm i}} (v+iHv)$, we see that $u$
satisfies
\begin{align} \label{complexNLH2}
u_t  = \Delta u  +u^2.
\end{align}
Noticing that $\widehat{Hv}(\xi) = -{\rm i} \, {\rm sgn}(\xi) \widehat{v}(\xi)$, we have\footnote{$\chi_E$ denote the characteristic function on $E$.}
\begin{align} \label{halflinesupport}
\widehat{u} = -{\rm i} \chi_{[0,\infty)} \widehat{v},
\end{align}
which means that $\widehat{u}$ is supported in the half line $\mathbb{R}_+=[0,\infty)$. So, it is natural to consider the solution of \eqref{complexNLH2} whose Fourier transform is supported in $\mathbb{R}_+$.  Guo, Ninomiya, Shimojo and Yanagida \cite{GuNiShYa2013} gave a systematical study to \eqref{complexNLH2} for the continuous and bounded initial data, they show that if the convex hull of the image of the initial data does not intersect the positive real axis, then the solution exists globally in time and converges to the trivial steady state. In one-dimensional space, they provide some solutions with nontrivial imaginary parts that blow up simultaneously. If the initial data are asymptotically constants, they also obtained some global existence and blowup results.  Chouichi,  Otsmane and Tayachi \cite{ChOtTa2015} considered a class of  continuous, bounded and decaying initial data like $\mathfrak{Re} u_0 \sim |x|^{-2\alpha}$, $\mathfrak{Im} u_0 \sim |x|^{-2\beta}$ (as $|x|\to \infty$) and they obtained the existence and the asymptotic behavior of global solutions. Harada \cite{Ha2016,Ha2017} constructed the non-simultaneous blow-up solutions in one spatial dimension.
 For the general complex-valued SLH \eqref{model},  Chouichi,  Majdoub and Tayachi \cite{ChMaTa2018} generalized the results in \cite{ChOtTa2015} to any power nonlinearity $u^m$, $2\leq m\in \mathbb{N}$.  The asymptotic behavior of the blow-up profiles at blow-up time were constructed in  Nouaili and Zaag \cite{NoZa2015} and Duong \cite{Du2019} for $m=2$ and for $m \geq 2$, respectively. For the irrational power $m >1$, an interesting  blowup complex solution with one blowup point was constructed in Duong \cite{Du2019a} and its asymptotic profiles were also obtained.  The local or global solutions obtained in \cite{ChOtTa2015,ChMaTa2018,Du2019,Du2019a,GuNiShYa2013,Ha2016,Ha2017,NoZa2015} are  in the space $L^\infty(\mathbb{R}^d) \cap C(\mathbb{R}^d)$ and the blowing up means that
 $$
 \lim\sup_{t\to T}\|u(t)\|_{L^\infty}=\infty
 $$
 for some $0<T<\infty.$

Now we state our main results and indicate the crucial ideas. One of our main results is the following: Let $f(u)=u^m$, $m\geq 2$, $s<0, \, \sigma \geq d/2-2/(m-1)$.  Assume that $u_0 \in E^s_\sigma$ with $ \widehat{u}_0 $ supported in the first octant away from the origin.\footnote{Such a kind of initial data have a direct relation with \eqref{halflinesupport} and they are complex-valued.}  Then there exists $s_0\leq s$ such that SLH \eqref{model} for has a unique solution $u \in C([0,\infty); E^{s_0}_\sigma) \cap \widetilde{L} ^m (0,\infty; E^{s_0, \sigma+2/m}_{2,2} )$, where $\widetilde{L} ^m (0,\infty; E^{s_0, \sigma+2/m}_{2,2} )$ is a  resolution space defined in \eqref{resolutionspace}. It seems that such a kind of results are a bit surprising, since we have no smallness condition on initial data and moreover, the initial data can be rougher than those in any Sobolev space $H^r$ with negative index $r <0$.

One of the main difficulties to solve SLH in supercritical Sobolev spaces $H^\sigma$ lies in the fact that one cannot make self-contained nonlinear mapping estimates and the ill-posedness occurs in supercritical Sobolev spaces $H^\sigma$. However, $E^s_\sigma$ type spaces have some good algebraic structures when the frequency is localized in the first octant, so that the nonlinear estimates become available in $E^s_\sigma$ type spaces. Using the exponential decay for the semigroup $e^{t\Delta}$, one can get a global well-posedness result of SLH for sufficiently small initial data in $E^s_\sigma$ for which their Fourier transforms are supported in the first octant away from the origin.

Recall that $ \dot H^{d/2-2/(m-1)}$ ($\dot H^\sigma= (-\Delta)^{-\sigma/2} L^2$) is  the scaling critical Sobolev space of SLH,  which means that the scaling solution $u_\lambda (t,x)= \lambda^{ 2/(m-1)} u(\lambda^2 t, \lambda x)$  is invariant in $\dot H^{d/2-2/(m-1)}$ for all $\lambda>0$.  Let us observe that $u_\lambda|_{t=0}$ in the supercritical space $ \dot H^\sigma$, for any $\sigma <d/2-2/(m-1)$,
\begin{align*}
\|u_\lambda|_{t=0}\|_{ \dot H^\sigma} =   \lambda^{2/(m-1)+ \sigma -d/2} \|u_0\|_{\dot H^\sigma} \to 0, \ \ \lambda \to \infty.
\end{align*}
It follows that the scaling solution can have very small initial data in supercritical Sobolev spaces. The above observation is also adapted to the supercritical space $E^s_\sigma$ ($s<0$),  $u_\lambda$ will vanish in $E^s_\sigma$  when $\lambda \to \infty$, which means that any large data in $E^s_\sigma$ can become small data by the scaling argument. This is why we can handle the large data in  $E^s_\sigma$ ($s<0$).

\subsection{Some notations and prelimilaries}

We denote by $L^p_x$ the Lebesgue space on $x\in \mathbb{R}^d$, i.e.,
$$
\|f\|_p :=\|f\|_{L^p_x} = \left(\int_{\mathbb{R}^d} |f(x)|^p dx \right)^{1/p}.
$$
For any function $g$ of $(t,x) \in \mathbb{R}_+ \times \mathbb{R}^d$, we denote
$$
\|g\|_{L^\gamma_t L^p_x(\mathbb{R}_+ \times \mathbb{R}^d)} = \|\|g\|_{L^p_x(\mathbb{R}^d)}\|_{L^\gamma_t(\mathbb{R}_+)},
$$
where $L^\gamma_t$ can be defined in a similar way as $L^p_x$ by replacing $\mathbb{R}^d$ with $\mathbb{R}_+ $. If there is no confusion, we will write $L^\gamma_t L^p_x =L^\gamma_t L^p_x(\mathbb{R}_+ \times \mathbb{R}^d)$.  Let us denote $\langle \nabla\rangle^s = \mathscr{F}^{-1} \langle \xi\rangle^s \mathscr{F}$, $2^{s|\nabla|}  = \mathscr{F}^{-1} 2^{s|\xi|} \mathscr{F}$.   We will use the frequency-uniform decomposition techniques, which were first applied to nonlinear PDE in \cite{WaZhGu2006}, see also some recent works \cite{BaVi2021,ChHuKuPa2019,Iw10,Ka2017,Ka2018,Pa2019,Pa2019a,SuWaZh15,WaHu07,WaHud07} and their references in the study for a variety of nonlinear evolution equations. For any $k\in \mathbb{Z}^d$, we denote
\begin{align}
\Box_k = \mathscr{F}^{-1} \chi_{k+[0,1)^d}\mathscr{F}, \label{wavepacketstran}
\end{align}
where $\chi_{A}$ denote the characteristic function on $A \subset \mathbb{R}^d$. $ \Box_k $ ($k\in \mathbb{Z}^d$) are said to be the frequency-uniform decomposition operators. In view of Plancherel's identity and the orthogonality of $\Box_k$, we see that
\begin{align} \label{equivnormes}
\|u\|_{E^s_\sigma} \sim  \left(\sum_{k\in \mathbb{Z}^d} 2^{2s|k|} \langle k\rangle^{2 \sigma} \|\Box_k u\|^2_2 \right)^{1/2}.
\end{align}
If $1 < p\leq q\leq \infty$, we have (cf. \cite{FeGrLiWa2021,WaZhGu2006})
\begin{align} \label{LqLpest}
\|\Box_k u\|_{L^q_x}  \lesssim  \|\Box_k u\|_{L^p_x}
\end{align}
holds for all $k\in \mathbb{Z}^d$. Since $\widehat{\Box_k u}$ is supported on $k+[0,1)^d$, we see that
\begin{align} \label{orthog}
\Box_k(\Box_{k_1} u ... \Box_{k_m} u) =0, \ \  |k-k_1-...-k_m|_\infty > m+1.
\end{align}
This is an important fact for the frequency-uniform decomposition, which will be repeatedly applied in the paper. For convenience, we denote for any $\mathbb{A} \subset \mathbb{Z}^d$,
\begin{align} \label{resolutionspace}
& \|u\|_{\widetilde{L} ^\gamma (0,\infty; E^{s,\sigma}_{p,q} (\mathbb{A})) } =  \left(\sum_{k\in \mathbb{A}} 2^{s|k|q} \langle k\rangle^{\sigma q} \|\Box_k u  \|^q_{L^\gamma_t L^p_x (\mathbb{R}_+\times\mathbb{R}^d)} \right)^{1/q}
\end{align}
and $\widetilde{L} ^\gamma (0,\infty; E^{s,\sigma}_{p,q} ) := \widetilde{L} ^\gamma (0,\infty; E^{s,\sigma}_{p,q} (\mathbb{Z}^d))$,  $\widetilde{L} ^\gamma (0,\infty; E^{s}_{p,q} ) := \widetilde{L} ^\gamma (0,\infty; E^{s,0}_{p,q})$.

 Throughout this paper, we write $|x|=|x(1)|+...+ |x(d)| $, $|x|_\infty = \max_{1\leq i\leq d} |x(i)|$ and $\langle x\rangle=(1+x(1)^2+...+x(d)^2)^{1/2}$ for  $x=(x(1),...,x(d)) \in \mathbb{R}^d$.  We will use the following notations. $C\ge 1, \ c\le 1$ will denote constants which can be different at different places, we will use $A\lesssim B$ to denote   $A\leqslant CB$; $A\sim B$ means that $A\lesssim B$ and $B\lesssim A$, $A\vee B= \max(A,B)$.   We denote by  $\mathscr{F}^{-1}f$ the inverse Fourier transform of $f$.  For any $1\leq p \leq \infty$, we denote by $p'$ the dual number of $p$, i.e., $1/p+1/p'=1$,  $\ell^p$  stands for the sequence Lebesgue space.   The following inequality will be used in this paper (cf. \cite{BeLo1976,WaHuHaGu2011}).

\begin{prop}\label{Nikolskii}
{\rm (Multiplier estimate)}  Let  $1 \leq r\le \infty, \ L \geq [d/2]+1$ and $\rho\in H^L$.   Then
we have
\begin{align} \label{1.35}
\|\mathscr{F}^{-1}\rho  \|_1  \lesssim \|\rho\|^{1-d/2L}_{2} \left(\sum^d_{i=1} \|\partial^L_{x_i}\rho \|_{2} \right)^{d/2L} .
\end{align}
\end{prop}

The paper is organized as follows. In Sections  \ref{sectNLH1D} and \ref{sectNLHhD} we consider the global existence and uniqueness  of solutions in $E^s_\sigma$ for the semi-linear heat equations with the nonlinearity $u^m$, $ m\geq 2$ by assuming that the Fourier transforms of initial data supported in the first octant and away from $0$, where Fujita's critical or subcritical powers are contained in 1D and 2D.  If $m\geq 2\vee ( 1+4/d)$, we can remove the condition that the Fourier transforms of initial data are away from the origin. In Section \ref{sectNLHexp}, we show the global existence and uniqueness of solutions in $E^s_\sigma$  for the semi-linear heat equations with an exponential nonlinearity $e^u-1$. In Section \ref{SLHmodspaces} we can generalize the above results to the spaces $E^s_{2,1}$.  Finally, we give  an example  to describe the solutions of the semi-linear heat equations.

\section{SLH in $E^s_\sigma$} \label{sectNLH1D}

In this section, we will consider \eqref{model} with $f(u)=u^m$ in the super-critical space $E^s_\sigma$. Denote
$$
\mathbb{R}^d_I := \{\xi\in \mathbb{R}^d :\  \xi(j) \geq 0, \ j=1,..., d\}, \ \  \mathbb{Z}^d_I = \mathbb{R}^d_I \cap \mathbb{Z}^d.
$$
We have the following

\begin{thm}\label{1DNLHglobal}
Let $d \geq 1$,  $f(u)=u^m$, $m\in \mathbb{N}\setminus \{1\}$. Let $s \leq 0, \ \sigma \geq d/2-2/(m-1)$,  $u_0\in E^{s}_\sigma$ with $\mathrm{supp}~\widehat{u}_0\subset \mathbb{R}^d_I  \setminus \{0\}$. We have the following results.
\begin{itemize}

\item[\rm (i)] If $s<0$, then there exists $s_0   \leq s$ such that SLH \eqref{model} has a unique solution $u\in \widetilde{L}^m (0,\infty; \, E^{s_0, \sigma+2/m}_{2,2} )$ satisfying  the equivalent integral equation
\begin{align} \label{1DNLHI}
u(t) = e^{t\Delta} u_0 + \int_0^t e^{(t-\tau)\Delta}u^m(\tau)~d\tau
\end{align}
and
\begin{align} \label{existspace}
u\in C([0,\infty); E^{s_0}_\sigma)\cap \widetilde{L}^\infty (0,\infty; \, E^{s_0, \sigma}_{2,2} ) \cap \widetilde{L}^1 (0,\infty; \, E^{s_0, 2+ \sigma}_{2,2} ).
\end{align}
Moreover, condition ${\rm supp}\, \widehat{u}_0 \subset \mathbb{R}^d_I$  is sharp in the sense that, for any $s_0 \leq s$, the solution map $u_0\to u$ from $E^s_{\sigma} $ into $  E^{s_0}_{\sigma}$  is not  $C^m$ for certain initial data whose Fourier transforms are supported in   $\mathbb{R}^d_I \cup (-\mathbb{R}^d_I)$.

\item[\rm (ii)] Let $s=0$.  Assume that  $\|u_0\|_{E^{0}_\sigma}$ is sufficiently small, then \eqref{1DNLHI} has a unique solution $u$ satisfying \eqref{existspace} for $s_0=0$. Moreover, $\sigma \geq \sigma_c = d/2-2/(m-1)$ is  optimal in the sense that the solution map   $u_0  \to  u$   is not  $C^m$  in $E^0_{\sigma,+}:= \{f\in E^0_\sigma: \, {\rm supp} \widehat{f} \subset \mathbb{R}^d_I \setminus \{0\}\}$ if $\sigma < \sigma_c$.

\item[\rm (iii)] Let $s<0$ and $s_0$ be as in (i). Let $\{u^{(j)}\}$ be the sequence of the iteration solutions
\begin{align} \label{iterationseqm}
u^{(j+1)}(t)=   e^{t\Delta} u_0 +   \int_0^t e^{(t-\tau)\Delta}u^{(j)}(\tau)^m d\tau, \ \ u^{(0)} =0.
\end{align}
Then for any $\tilde{s}_0 < s_0$, there exists $C>1$ such that
$$
\|u^{(j)}(t)-u(t) \|_{E^{\tilde{s}_0}} \leq \frac{C^j}{(j\, !)^2}, \ \ \forall \  t \geq 0, \ j\in \mathbb{N}.
$$

\end{itemize}

 \end{thm}

Theorem \ref{1DNLHglobal} needs several remarks.

\begin{itemize}

\item[(i)] From condition $\mathrm{supp}~\widehat{u}_0\subset \mathbb{R}^d_I  \setminus \{0\}$, we see that there exists $\varepsilon_0 >0$ such that $\mathrm{supp}~\widehat{u}_0\subset \mathbb{R}^d_I  \setminus \{\xi: \, |\xi|< \varepsilon_0\} $. If $s<0$, $s_0:= s_0(\varepsilon_0, m, \sigma,s, \|u_0\|_{E^s_\sigma}) \leq s$ comes from the scaling argument of the solutions.  For example, in the scaling critical case $\sigma = d/2- 2/(m-1)$, we can take
$$
s_0  =s -( C +  \log_2^{1+ \|u_0\|_{E^s_\sigma}})   \varepsilon^{-1}_0 ,
$$
 see Remark \ref{s0exp} for details. However, if $\|u_0\|_{E^s_\sigma}$ is sufficiently small and ${\rm supp}\, \widehat{u}_0 \subset \mathbb{R}^d_I \setminus [0,1)^d$, then we can take $s_0=s$ in \eqref{existspace}.

\item[(ii)]  For the global existence and uniqueness results in Theorem \ref{1DNLHglobal}, we have no condition on the size of $u_0 \in E^s_\sigma$ and as indicated in Proposition \ref{embedding},  $E^s_\sigma$ is a rather rough space. Recall that condition  $\mathrm{supp}~\widehat{u}_0\subset \mathbb{R}^d_I\setminus \{0\}$ implies that $u_0$ is a complex valued function, such a kind of initial data are different from the nonnegative data in \cite{Fu1966,Ha1973,We1981}, also different from the initial data $\psi$ in \cite{CaDiWe2009}. For instance, let $d=1$, for any $k\in \mathbb{Z}_+$,
\begin{align}
    u_0 &= A\, e^{\mathrm{i}x}  \frac{d^k}{dx^k}\left(\delta (x) + \frac{2 \, \mathrm{ i}}{x}\right),  \  A\in \mathbb{C}, \label{example1}\\
   u_0 & = e^{\mathrm{i}x} \sum^\infty_{m=0} \frac{\lambda^m}{m!} (-{\rm i})^m \frac{d^k}{dx^k} \left(\delta (x) + \frac{2 \, \mathrm{ i}}{x}\right), \ \ |\lambda|< |s| . \label{example2}
\end{align}
 satisfy the condition of Theorem \ref{1DNLHglobal}, where $\delta$ is the Dirac measure.

 In \cite{GuNiShYa2013}, Guo, Ninomiya, Shimojo and Yanagida considered the SLH with nonlinearity $u^2$ and they showed that if $u_0 \in L^\infty (\mathbb{R}^d) \cap C(\mathbb{R}^d)$, $\mathfrak{Re} \, u_0 (x) < A \, \mathfrak{Im} \, u_0 (x)$ for some $A>0$ and for all $x\in \mathbb{R}^d$, then SLH has a unique global solution in $\in L^\infty (\mathbb{R}) \cap C(\mathbb{R})$. In 1D case, they obtained the blowup solution if either $\mathfrak{Re} \widehat{u}_0 (\xi)>0$  and $\mathfrak{Im} \widehat{u}_0 (\xi)=0$ for all $\xi \in \mathbb{R}$; or $\mathfrak{Re} u$ even, $\mathfrak{Im}u$ odd and $\mathfrak{Im}\, u_0(x)>0$ for all $x>0$. In  Theorem \ref{1DNLHglobal} we need $\widehat{u}_0(\xi) =0$ for any $\xi\in \mathbb{R}^d\setminus \mathbb{R}^d_I$, which has no implicit relations with pointwise condition $\mathfrak{Re} \, u_0   < A \, \mathfrak{Im} \, u_0 $ in \cite{GuNiShYa2013} and the  decaying condition $\mathfrak{Re} u_0 \sim c |x|^{-2\alpha}$, $\mathfrak{Im} u_0 \sim c |x|^{-2\beta}$ (as $|x|\to \infty$) in \cite{ChOtTa2015,ChMaTa2018}. Moreover, one easily sees that $\widehat{u}_0$ supported in $\mathbb{R}^d_I$ does not contradict  the blowup conditions in \cite{GuNiShYa2013}.

\item[(iii)] Recall that $E^0_{\sigma_c} = H^{\sigma_c}$ with $\sigma_c = d/2- 2/(m-1)$ is the scaling critical space of the SLH.  If ${\rm supp} \, \widehat{u}_0$ can be any subset of the whole line $ \mathbb{R}$ and $m\geq 3$, Molinet, Ribaud and Youssfi \cite{MoRiYo2002} showed that SLH in 1D is locally well-posed in $H^\sigma$ for $\sigma\geq \sigma_c$ and  ill-posed in $H^\sigma$ for $\sigma < \sigma_c$.  In the case $m=2$, they obtained that SLH in 1D is local well-posed in $H^\sigma$ for $\sigma >-1$ and ill-posed in $H^\sigma$ for $\sigma <-1$, and the local well-posedness in the critical case $\sigma=-1$ was shown by Molinet and Tayachi \cite{MoTa2015}, where the fractional NLH was also studied. However, in the case $m=2$ in 1D, our results indicate that the critical space is $H^{-3/2}$ if the Fourier transforms of solutions are supported in $\mathbb{R}^d_I$, which is different from $H^{-1}$ in \cite{MoRiYo2002,MoTa2015} (see also \cite{Iw10}). If $m\geq 1+4/d$, the existence of small data global solutions in (ii) of Theorem \ref{1DNLHglobal} seems known, which can be derived by following the global well-posedness in $L^r$ in Giga \cite{Gi1986} together with the standard regularity arguments.

\end{itemize}

\subsection{Linear and multi-linear estimates}

In Theorem \ref{1DNLHglobal}, one needs that the Fourier transforms of initial data and solutions supported in $\mathbb{R}^d_I \setminus \{0\}$. By scaling argument, we can first assume that
\begin{align} \label{rdplus}
{\rm supp}\, \widehat{u}_0 \subset \mathbb{R}^d_{I,+} : = \left\{\xi\in \mathbb{R}^d_I: \max_{1\leq j\leq d}\xi(j)\geq 1 \right\}.
\end{align}
The semi-group $e^{t\Delta}$ has very good regularity estimates in Besov and Triebel spaces, see for instance Iwabuchi and Nakamura \cite{IwNa2013}, Ogawa and Shimizu \cite{OgSh2010,OgSh2020}, Kozono, Okada and Shimizu \cite{KoOkSh2020,KoSh2019} and in modulation spaces \cite{Iw10, WaZhGu2006}. Concerning the very rough data, we have

\begin{lemma}\label{linearestimate}
Let $d\geq 1$, $s\leq 0$, $\sigma \in \mathbb{R}$, $1\leq \gamma_1\leq \gamma \leq \infty$.  Suppose that $\mathrm{supp}~ \mathscr{F}_x f(t), \, {\rm supp}\, \widehat{u}_0\subset \mathbb{R}^d_{I,+}$. Then we have
		\begin{align}
			\|e^{t\Delta}u_0\|_{ \widetilde{L} ^\gamma (0,\infty; E^{s,\sigma+2/\gamma}_{2,2})  } & \lesssim \|u_0\|_{E^s_\sigma},  \label{basicest1} \\
			\left\|\int_0^t e^{(t-\tau)\Delta}f(\tau)~d\tau\right\|_{\widetilde{L} ^\gamma (0,\infty; E^{s,\sigma+2/\gamma}_{2,2}) } &\lesssim \|f\|_{\widetilde{L} ^{\gamma_1} (0,\infty; E^{s,\sigma-2/\gamma'_1}_{2,2}) }.  \label{basicest2}
		\end{align}
	\end{lemma}
	\begin{proof}
By Plancherel's identity and {\rm supp}$\, \widehat{u}_0 \subset \mathbb{R}^d_{I,+}$, we have
\begin{align} \label{L2estimate}
\|\Box_k e^{t\Delta}u_0\|_2 \leq e^{- t|k|^2} \|\Box_k u_0\|_2
\end{align}	
for all $k\in \mathbb{Z}^d_I$. Hence, in view of \eqref{L2estimate} and H\"older's inequality,
\begin{align}
 \|\Box_k e^{t\Delta}u_0\|_{L^\gamma_t L_x^2 (\mathbb{R}_+ \times \mathbb{R}^d)  }
&\lesssim  |k|^{-2/\gamma}  \|\Box_k u_0\|_{2}, \ \ k\in \mathbb{R}^d_{I,+}\cap \mathbb{Z}^d_I. \label{L2estimate2}
\end{align}
\eqref{L2estimate2} is multiplied by $2^{s|k|} \langle k\rangle^{2/\gamma +\sigma}$ and then taken $\ell^2$-norm on $k\in \mathbb{Z}^d$, we have \eqref{basicest1}.  For the inhomogenous part, by \eqref{L2estimate} and Young's inequality, we have
\begin{align}
  \left\|\Box_k \int_0^t e^{(t-\tau)\Delta}f(\tau)~d\tau\right\|_{L^\gamma_tL^2_x (\mathbb{R}_+ \times \mathbb{R}^d) }
&\lesssim  \left\|\int_0^t e^{-(t-\tau)|k|^2}\|\Box_k f(\tau)\|_{L^2}d\tau\right\|_{L^\gamma_{t }(\mathbb{R}_+ )}  \nonumber\\
&\lesssim   |k|^{-2/\gamma-2/\gamma'_1}\|\Box_k f \|_{L^{\gamma_1}_t L_x^2 (\mathbb{R}_+ \times \mathbb{R}^d) }.  \label{basicest3}		
\end{align}
\eqref{basicest3} is multiplied by $2^{s|k|} \langle k\rangle^{2/\gamma +\sigma}$ and then taken $\ell^2$-norm on $k\in \mathbb{Z}^d$, we have \eqref{basicest2}.
	\end{proof}

\begin{cor}\label{linearestimateapp}
Let $d\geq 1$, $s\leq 0$, $\sigma \in \mathbb{R}$, $1\leq \gamma<\infty$.  Suppose that $\mathrm{supp}\mathscr{F}_x f(t), \, {\rm supp}\, \widehat{u}_0\subset \mathbb{R}^d_{I,+}$. Then we have
		\begin{align}
			\|e^{t\Delta}u_0\|_{ \widetilde{L} ^\gamma (0,\infty; E^{s,\sigma+2/\gamma}_{2,2})  \cap \widetilde{L} ^\infty (0,\infty; E^{s,\sigma}_{2,2})} & \lesssim \|u_0\|_{E^s_\sigma},  \label{corbasicest1} \\
			\left\|\int_0^t e^{(t-\tau)\Delta}f(\tau)~d\tau\right\|_{\widetilde{L} ^\gamma (0,\infty; E^{s,\sigma+2/\gamma}_{2,2})  \cap \widetilde{L} ^\infty (0,\infty; E^{s,\sigma}_{2,2}) } &\lesssim \|f\|_{\widetilde{L} ^1 (0,\infty; E^{s,\sigma}_{2,2}) }.  \label{corbasicest2}
		\end{align}
	\end{cor}

\begin{lemma}\label{interlinearestimate}
Let $d\geq 1$, $s\leq 0$, $\sigma \in \mathbb{R}$, $1\leq \gamma<\infty$.    Then we have $\widetilde{L} ^1 (0,\infty; E^{s,\sigma+2}_{2,2})  \cap \widetilde{L} ^\infty (0,\infty; E^{s,\sigma}_{2,2})    \subset \widetilde{L} ^\gamma (0,\infty; E^{s,\sigma+2/\gamma}_{2,2})$ and
\begin{align}
\|u\|_{\widetilde{L} ^\gamma (0,\infty; E^{s,\sigma+2/\gamma}_{2,2})}	\leq  \|u\|^{1/\gamma}_{\widetilde{L} ^1 (0,\infty; E^{s,\sigma+2 }_{2,2})}		\|u\|^{1-1/\gamma}_{\widetilde{L} ^\infty(0,\infty; E^{s,\sigma}_{2,2})}. 	
			\label{intercorbasicest2}
\end{align}
	\end{lemma}
\begin{proof}
It is a consequence of H\"older's inequality.
\end{proof}

\begin{lemma}\label{nonlinearestimate}
Let $ m \geq 2$, $s\leq 0$, $\sigma \geq d/2 -2/(m-1)$. Suppose that $\mathrm{supp} \mathscr{F}_x u_j (t) \subset \mathbb{R}^d_{I}, ~1\leq j\leq m, \ t\geq 0$. Then we have, for some $C_{m}>0$,
\begin{align}
\|u_1u_2\cdots u_m\|_{\widetilde{L} ^1 (0,\infty; E^{s,\sigma}_{2,2}) }\leq C_{m} \prod^m_{j=1} \|u_j\|_{ \widetilde{L} ^m (0,\infty; E^{s,\sigma+2/m}_{2,2}) }, \label{1nonest}
\end{align}
where $C_{m} \leq C^m m^{m/2}$ if $\sigma \geq d/2$, $C$ depends only on $s,d,\sigma$.
\end{lemma}
\begin{proof}
By definition and the support set property of $\mathscr{F}u_j (t)$, we have
\begin{align}
\|u_1u_2\cdots u_m\|_{\widetilde{L} ^1 (0,\infty; E^{s,\sigma}_{2,2}) } = \left( \sum_{k\in \mathbb{Z}^d_I } \langle k\rangle^{2\sigma} 2^{2s|k|}  \|\Box_k(u_1u_2...u_m)\|^2_{L^1_t L^2_x (\mathbb{R}_+ \times \mathbb{R}^d) } \right)^{1/2}. \label{1nonest}
\end{align}
Using the definition of $\Box_k$, we see that
\begin{align*}
 \|\Box_k(u_1u_2\cdots u_m)\|_{L_t^1L_x^2}
			&\leq \sum_{k_1,...,k_m \in \mathbb{Z}^d_I }   \|\Box_k (\Box_{k_1}u_1\Box_{k_2}u_2\cdots\Box_{k_m} u_m)\|_{L_{t}^1L_x^2}.
\end{align*}
By the orthogonality of $\Box_k$,
$$
\Box_k (\Box_{k_1}u_1\Box_{k_2}u_2\cdots\Box_{k_m} u_m) \neq 0,
$$
implies that $(k_1,...,k_m)$ belongs to
$$
\Lambda_k (k_1,...,k_m) : = \left\{(k_1,...,k_m): \,  -1 \leq k(j) - k_1(j)-...-k_m(j) \leq m, \ \ j=1,...,d \right\}.
$$
It follows that
\begin{align}
 \|\Box_k(u_1u_2\cdots u_m)\|_{L_t^1L_x^2}
			&\leq \sum_{k_1,...,k_m \in \mathbb{Z}^d_I }   \|\Box_{k_1}u_1\Box_{k_2}u_2\cdots\Box_{k_m} u_m \|_{L_{t}^1L_x^2} \chi_{\Lambda_k (k_1,...,k_m)}. \label{3nonest}
\end{align}
Let us further write
$$
(\mathbb{Z}^d_{I})^{m,i} : = \left\{(k_1,...,k_m) \in (\mathbb{Z}^d_{I})^m : \, |k_i| = \max_{1\leq j \leq m} |k_j | \right\}.
$$
One has that
\begin{align}
 \|\Box_k(u_1u_2\cdots u_m)\|_{L_t^1L_x^2}
			&\leq \sum^m_{i=1}  \sum_{(k_1,...,k_m) \in (\mathbb{Z}^d_{I})^{m,i} }   \|\Box_{k_1}u_1\Box_{k_2}u_2\cdots\Box_{k_m} u_m \|_{L_{t}^1L_x^2} \chi_{\Lambda_k (k_1,...,k_m)} \nonumber\\
& : = \sum^m_{i=1} L_i.  \label{4nonest}
\end{align}
By symmetry, it suffices to estimate $L_m$. By H\"older's inequality and $\|\Box_k u\|_\infty \leq \|\Box_k u\|_2$,
\begin{align}
 L_m & \leq    \sum_{(k_1,...,k_m) \in (\mathbb{Z}^d_{I})^{m,m} }  \prod^{m}_{j=1}  \|\Box_{k_j}u_j\|_{L^m_{t}L^\infty_x} \|\Box_{k_m} u_m \|_{L_{t}^m L_x^2} \chi_{\Lambda_k (k_1,...,k_m)} \nonumber\\
&  \leq    \sum_{(k_1,...,k_m) \in (\mathbb{Z}^d_{I})^{m,m} }  \prod^{m }_{j=1}  \|\Box_{k_j}u_j\|_{L^m_{t}L^2_x}   \chi_{\Lambda_k (k_1,...,k_m)}.   \label{5nonest}
\end{align}
Let us  write
$$
(\mathbb{Z}^d_{I})^{m-1}_{r}  : = \left\{(k_1,...,k_{m-1}) \in (\mathbb{Z}^d_{I})^{m-1}  : \, |k_1|,...,|k_{m-1}| \leq  r \right\}.
$$
Then for $l_m =k-k_1-...-k_{m-1}$, $\mathbb{A} = \{l\in \mathbb{Z}^d:\ -1 \leq l(j) \leq m, \, j=1,...,d\}$,
\begin{align}
 L_m
&  \leq   \sum_{l\in \mathbb{A} }  \sum_{(k_1,...,k_{m-1}) \in (\mathbb{Z}^d_{I})^{m-1}_{|l_m-l|} }  \prod^{m-1}_{j=1}  \|\Box_{k_j}u_j\|_{L^m_{t}L^2_x}  \|\Box_{l_m-l}u_m\|_{L^m_{t}L^2_x}\chi_{\{l_m-l\in \mathbb{Z}^d_I\}}   .   \label{5anonest}
\end{align}
Using the Cauchy-Schwarz inequality, one has that
\begin{align}
 L_m
\leq &  \sum_{l\in \mathbb{A} } \left( \sum_{(k_1,...,k_{m-1}) \in (\mathbb{Z}^d_{I})^{m-1}_{|l_m-l|} }  \prod^{m-1}_{j=1} \langle k_j\rangle^{2(\sigma+2/m)} 2^{2s|k_j|}  \|\Box_{k_j}u_j\|^2_{L^m_{t}L^2_x} \right)^{1/2} \nonumber\\
&  \times   \left( \sum_{(k_1,...,k_{m-1}) \in (\mathbb{Z}^d_{I})^{m-1}_{|l_m-l|} }  \prod^{m-1}_{j=1} \langle k_j\rangle^{-2(\sigma+2/m)} 2^{-2s|k_j|}
 \|\Box_{l_m-l} u_m \|^2_{L_{t}^m L_x^2} \chi_{\{l_m-l\in \mathbb{Z}^d_I\}}   \right)^{1/2} \nonumber\\
 \lesssim  &  \  2^{|s|dm} 2^{-s|k|}  \prod^{m-1}_{j=1}  \|u_j\|_{\widetilde{L} ^m (0,\infty; E^{s,\sigma+2/m}_{2,2}) } \nonumber\\
&  \times \sum_{l\in \mathbb{A} }  \left( \sum_{(k_1,...,k_{m-1}) \in (\mathbb{Z}^d_{I})^{m-1}_{|l_m-l|} }  \prod^{m-1}_{j=1} \langle k_j\rangle^{-2(\sigma+2/m)} \left(2^{s|l_m-l|} \|\Box_{l_m-l} u_m \|_{L_{t}^m L_x^2} \right)^2 \right)^{1/2}.
 \label{6nonest}
\end{align}
By \eqref{6nonest}, we have for $l_m =k-k_1-...-k_{m-1}$,
\begin{align}
&  \left( \sum_{k\in \mathbb{Z}^d_I} \langle k \rangle^{2\sigma} 2^{2s|k|}  L_m^2 \right)^{1/2} \nonumber\\
 &  \lesssim \  2^{|s|dm}  \sum_{l\in \mathbb{A} } \left(  \sum_{k\in \mathbb{Z}^d_I} \langle k \rangle^{2\sigma}  \sum_{ (k_1,...,k_{m-1}) \in (\mathbb{Z}^d_{I})^{m-1}_{|l_m-l|}  }  \prod^{m-1}_{j=1} \langle k_j\rangle^{-2(\sigma+2/m)}
  2^{2 s|l_m-l|} \|\Box_{l_m-l} u_m \|_{L_{t}^m L_x^2}^2  \right)^{1/2}  \nonumber\\
&  \quad\quad \times \prod^{m-1}_{j=1}  \|u_j\|_{\widetilde{L} ^m (0,\infty; E^{s,\sigma+2/m}_{2,2}) }  .   \label{7nonest}
\end{align}
Noticing that $|l|\leq dm$  for $l\in \mathbb{A}$, we have
\begin{align}
  \sum_{k\in \mathbb{Z}^d_I} & \langle k \rangle^{2\sigma}  \sum_{ (k_1,...,k_{m-1}) \in (\mathbb{Z}^d_{I})^{m-1}_{|l_m-l|}  }  \prod^{m-1}_{j=1} \langle k_j\rangle^{-2(\sigma+2/m)}
  2^{2 s|l_m-l|} \|\Box_{l_m-l} u_m \|_{L_{t}^m L_x^2}^2   \nonumber\\
 & \leq  \sum_{k_m \in \mathbb{Z}^d_I}  \sum_{(k_1,...,k_{m-1}) \in (\mathbb{Z}^d_{I})^{m-1}_{|k_m|}  } \left\langle \sum^m_{j=1} k_j+l \right\rangle^{2\sigma} \prod^{m-1}_{j=1} \langle k_j\rangle^{-2(\sigma+2/m)}  \left( 2^{s|k_m|} \|\Box_{k_m} u_m \|_{L_{t}^m L_x^2}\right)^2    \nonumber\\
& \lesssim m^{2|\sigma|} \sum_{k_m \in \mathbb{Z}^d_I}  \sum_{(k_1,...,k_{m-1}) \in (\mathbb{Z}^d_{I})^{m-1}_{|k_m|}  }  \prod^{m-1}_{j=1} \langle k_j\rangle^{-2(\sigma+2/m)} \langle k_m \rangle^{-4/m}\nonumber\\
 & \quad\quad\quad \quad\quad\quad \quad\quad\quad \quad\quad\quad  \times \left(\langle k_m \rangle^{\sigma +2/m} 2^{s|k_m|} \|\Box_{k_m} u_m \|_{L_{t}^m L_x^2}\right)^2.     \label{8nonest}
\end{align}
If $d/2-2/(m-1) \leq \sigma <d/2 -2/m$, then we easily see that, for $\sigma_c=d/2-2/(m-1)$,
\begin{align}
\sum_{ (k_1,...,k_{m-1}) \in (\mathbb{Z}^d_{I})^{m-1}_{|k_m|} } \prod^{m-1}_{j=1} \langle k_j\rangle^{-2(\sigma+2/m)} \langle  k_m \rangle^{-4/m} & \leq \frac{C^{m-1}}{(d-2(\sigma+2/m))^{m-1}} \langle  k_m \rangle^{(m-1)(d-2\sigma_c)-4}  \nonumber\\
 & \leq \frac{C^{m-1}}{|d-2(\sigma+2/m)|^{m-1}} .  \label{const1}
\end{align}
If $ \sigma =d/2 -2/m$, we have
\begin{align}
\sum_{ (k_1,...,k_{m-1}) \in (\mathbb{Z}^d_{I})^{m-1}_{|k_m|}  } \prod^{m-1}_{j=1} \langle k_j\rangle^{-2(\sigma+2/m)} \langle  k_m \rangle^{-4/m} & \leq \sum_{(k_1,...,k_{m-1}) \in (\mathbb{Z}^d_{I})^{m-1}_{|k_m|}  } \prod^{m-1}_{j=1} \langle k_j\rangle^{-d-4/m(m-1)}  \nonumber \\
 & \leq C^{m-1} (m(m-1))^{m-1}. \label{const2}
\end{align}
If $d/2 -2/m < \sigma  <d/2$, we have
\begin{align}
\sum_{ (k_1,...,k_{m-1}) \in (\mathbb{Z}^d_{I})^{m-1}_{|k_m|} } \prod^{m-1}_{j=1} \langle k_j\rangle^{-2(\sigma+2/m)} \langle  k_m \rangle^{-4/m} & \leq
 \frac{C^{m-1}}{|d-2(\sigma+2/m)|^{m-1}}. \label{const3}
\end{align}
If $\sigma=d/2$, then we have
\begin{align}
\sum_{ (k_1,...,k_{m-1}) \in (\mathbb{Z}^d_{I})^{m-1}_{|k_m|} } \prod^{m-1}_{j=1} \langle k_j\rangle^{-2(\sigma+2/m)} \langle  k_m \rangle^{-4/m} & \leq
(C m)^{m-1}. \label{const4}
\end{align}
If $\sigma> d/2$, then we have
\begin{align}
\sum_{(k_1,...,k_{m-1}) \in (\mathbb{Z}^d_{I})^{m-1}_{|k_m|} } \prod^{m-1}_{j=1} \langle k_j\rangle^{-2(\sigma+2/m)} \langle  k_m \rangle^{-4/m} & \leq
C^m. \label{const5}
\end{align}
So, \eqref{8nonest} has a bound
\begin{align}
 \sum_{k\in \mathbb{Z}^d_I} & \langle k \rangle^{2\sigma}  \sum_{ (k_1,...,k_{m-1}) \in (\mathbb{Z}^d_{I})^{m-1}_{|l_m-l|}  }  \prod^{m-1}_{j=1} \langle k_j\rangle^{-2(\sigma+2/m)}
  2^{2 s|l_m-l|} \|\Box_{l_m-l} u_m \|_{L_{t}^m L_x^2}^2  \nonumber\\
  &\lesssim m^{2|\sigma|} C^m A_m  \|u_m \|^2_{\widetilde{L} ^m (0,\infty; E^{s,\sigma+2/m}_{2,2})},       \label{9nonest}
\end{align}
where $A_m$ is the number in the right hand sides of \eqref{const1}--\eqref{const5} corresponding to different  $\sigma\geq d/2-2/(m-1)$.  Inserting \eqref{9nonest} into \eqref{7nonest} and noticing that $m^{2|\sigma|} \leq C^m$, $\mathbb{A}$ has at most $d^m$ many indices, one has that for $C_m = C^m A_m^{1/2}$,
\begin{align}
  \left(\sum_{k\in \mathbb{Z}^d_I} \langle k \rangle^{2\sigma} 2^{2s|k|}  L_m^2 \right)^{1/2}
 \leq C_{ m }  \prod^{m}_{j=1}  \|u_j\|_{\widetilde{L} ^m (0,\infty; E^{s,\sigma+2/m}_{2,2}) } .   \label{10nonest}
\end{align}
In an analogous way to \eqref{10nonest}, one can estimate
\begin{align}
   \left( \sum_{k\in \mathbb{Z}^d_I} \langle k \rangle^{2\sigma} 2^{2s|k|}  L_i^2 \right)^{1/2}
 \leq  C_{  m } \prod^{m}_{j=1 }  \|u_j\|_{\widetilde{L} ^m (0,\infty; E^{s,\sigma+2/m}_{2,2}) }, \ \ i=1,...,m-1  .   \label{11nonest}
\end{align}
By \eqref{const4}, \eqref{const5}, we see that $C_m \leq C^m m^{m/2}$ if $\sigma \geq d/2$.  Combining \eqref{1nonest}, \eqref{4nonest}, \eqref{10nonest} and \eqref{11nonest},  we obtain the result, as desired.
\end{proof}

\begin{rem} \label{appendnonlinearest}
From the proof of Lemma \ref{nonlinearestimate}, we see that if we remove the lower frequency part of $u_1...u_m$, we can obtain that
\begin{align}
  \|u_1u_2\cdots u_m &  -\Box_0 u_1\, \Box_0 u_2...\Box_0 u_m \|_{\widetilde{L} ^1 (0,\infty; E^{s,\sigma}_{2,2}(\mathbb{Z}^d\setminus \{0\})) } \nonumber\\
  & \leq C_m \sum^m_{i=1} \|u_i\|_{ \widetilde{L} ^1 (0,\infty; E^{s,\sigma+2}_{2,2}(\mathbb{Z}^d\setminus \{0\}) ) } \prod^m_{j=1, \,j\neq i} \|u_j\|_{  \widetilde{L} ^\infty (0,\infty; E^{s,\sigma}_{2,2}) }, \label{wl1nonest}
\end{align}
where the constant $C_m =C^{m} m^{m/2}$ for $\sigma =d/2$, and $C_m =C^m $ for $\sigma \neq d/2$.
Indeed,  we write
 $$
 (\mathbb{Z}^d_{I})^m_+ = \{(k_1,...,k_m)\in (\mathbb{Z}^d_{I})^m: \, \max_{1\leq j\leq m} |k_j|_\infty \geq 1 \}.
 $$
 By the definition, we have
\begin{align}
\|u_1 & u_2\cdots u_m - \Box_0 u_1\, \Box_0 u_2...\Box_0 u_m \|_{\widetilde{L} ^1 (0,\infty; E^{s,\sigma}_{2,2}(\mathbb{Z}^d_I \setminus \{0\})) }  \nonumber\\
& = \left( \sum_{k\in \mathbb{Z}^d_I\setminus \{0\} } \langle k\rangle^{2\sigma} 2^{2s|k|}  \|\Box_k(u_1u_2...u_m- \Box_0 u_1\, \Box_0 u_2...\Box_0 u_m )\|^2_{L^1_t L^2_x (\mathbb{R}_+ \times \mathbb{R}^d) } \right)^{1/2}. \label{w1nonest}
\end{align}
Observing that
\begin{align}
   u_1u_2\cdots u_m &  -\Box_0 u_1\, \Box_0 u_2...\Box_0 u_m  = \sum_{(k_1,...,k_m) \in (\mathbb{Z}^d_{I})^m_+ } \Box_{k_1} u_1\, \Box_{k_2} u_2...\Box_{k_m} u_m,  \label{wl1nonest2}
\end{align}
we have
\begin{align}
 \|\Box_k & (u_1u_2\cdots u_m -\Box_0 u_1 \Box_0 u_2 ...\Box_0 u_m )\|_{L_t^1L_x^2} \nonumber\\
			&\leq \sum^m_{i=1}  \sum_{(k_1,...,k_m) \in (\mathbb{Z}^d_{I})^{m,i}_{+} }   \|\Box_{k_1}u_1\Box_{k_2}u_2\cdots\Box_{k_m} u_m \|_{L_{t}^1L_x^2} \chi_{\Lambda_k (k_1,...,k_m)} \nonumber\\
& : = \sum^m_{i=1} L_i.  \label{w4nonest}
\end{align}
where $(\mathbb{Z}^d_{I})^{m,i}_{+} =\{ (k_1,...,k_m) \in (\mathbb{Z}^d_{I})^m:\, |k_i|_\infty = \max_{1\leq j\leq m} |k_j|_\infty  \geq 1   \}$.
One can estimate $L_m$ by
\begin{align}
 L_m
&  \leq    \sum_{(k_1,...,k_m) \in (\mathbb{Z}^d_{I})^{m,m}_{+} }  \prod^{m-1}_{j=1}  \|\Box_{k_j}u_j\|_{L^\infty_{t}L^2_x}  \|\Box_{k_m}u_m\|_{L^1_{t}L^2_x}  \chi_{\Lambda_k (k_1,...,k_m)},   \label{w5nonest}
\end{align}
In this way we have removed the lower frequency part of $u_m$ in the estimate of $L_m$. Using a similar way as in the proof of Lemma \ref{nonlinearestimate}, we can get \eqref{wl1nonest}.
\end{rem}

\begin{lemma} \label{scaling}
		Assume $s\leq 0$, $\varphi \in E^s_\sigma (\mathbb{R}^d)$.  Write  $ D_\lambda: \, \varphi \to \varphi_\lambda (\cdot) = \varphi(\lambda\cdot)$. Assume that ${\rm supp }\, \varphi \subset \{\xi: |\xi|\geq \varepsilon_0\}$ for some $\varepsilon_0>0$. Then  for any $\lambda> 1$, we have
\begin{align*}
& 			\|D_\lambda \varphi \|_{E^s_\sigma (\mathbb{R}^d) } \lesssim \lambda^{-d/2} 2^{s (\lambda-1) \varepsilon_0} \|\varphi \|_{E^{s}_\sigma (\mathbb{R}^d) }, \ \ \sigma \leq 0 ; \\
& 			\|D_\lambda \varphi \|_{E^s_\sigma (\mathbb{R}^d) } \lesssim \lambda^{-d/2+\sigma} 2^{s (\lambda-1) \varepsilon_0} \|\varphi \|_{E^{s}_\sigma (\mathbb{R}^d) }, \ \ \sigma > 0.
\end{align*}
For any $\mu<1$, we have
\begin{align*}
& \|D_\mu \varphi \|_{E^{s/\mu}_\sigma (\mathbb{R}^d) } \lesssim \mu^{-d/2+\sigma}  \|\varphi \|_{E^{s}_\sigma (\mathbb{R}^d) }, \ \ \sigma \leq 0; \\
& 	\|D_\mu \varphi \|_{E^{s/\mu}_\sigma (\mathbb{R}^d) } \lesssim \mu^{-d/2} \|\varphi \|_{E^{s}_\sigma (\mathbb{R}^d) }, \ \ \sigma > 0.
\end{align*}
\end{lemma}
	\begin{proof}[\textit{Proof}]
		By the definition of $E^s_\sigma (\mathbb{R}^d)$, we have
		\begin{align*}
			\|\varphi_\lambda\|_{E^s_\sigma}
			 = \lambda^{-d/2}\|\langle \lambda \xi\rangle^\sigma 2^{s\lambda |\xi|}\widehat{\varphi}(\xi)\|_{L^2_{\xi}}.
		\end{align*}
If  ${\rm supp } \, \widehat{\varphi}  \subset \{\xi: |\xi|\geq \varepsilon_0\}$ and $\sigma>0$, then
		\begin{align*}
			\|\varphi_\lambda\|_{E^s_\sigma }
			& \leq  \lambda^{-d/2 +\sigma }\|\langle   \xi\rangle^\sigma 2^{s\lambda |\xi|}\widehat{\varphi}(\xi)\|_{L^2_{\xi}}
			  \leq \lambda^{-d/2+\sigma } 2^{s (\lambda-1) \varepsilon_0} \|\varphi\|_{E^{s}_\sigma }.
		\end{align*}
If  ${\rm supp } \, \widehat{\varphi}  \subset \{\xi: |\xi|\geq \varepsilon_0\}$ and $\sigma\leq 0$, then
		\begin{align*}
			\|\varphi_\lambda\|_{E^s_\sigma }
			& \leq  \lambda^{-d/2 }\|\langle   \xi\rangle^\sigma 2^{s\lambda |\xi|}\widehat{\varphi}(\xi)\|_{L^2_{\xi}}
			  \leq \lambda^{-d/2} 2^{s (\lambda-1) \varepsilon_0} \|\varphi\|_{E^{s}_\sigma }.
		\end{align*}
Hence, we have the results for $\lambda >1$. For $\mu <1$, we have
\begin{align*}
			\|\varphi_\mu\|_{E^{s/\mu}_\sigma }
			 =   \mu^{-d/2 }\|\langle  \mu \xi\rangle^\sigma 2^{s |\xi|}\widehat{\varphi}(\xi)\|_{L^2_{\xi}}
		\end{align*}
Using $\langle  \mu \xi\rangle \geq \mu \langle \xi\rangle $ for $\sigma \leq 0$, and  $\langle  \mu \xi\rangle \leq   \langle \xi\rangle $ for $\sigma > 0$, we have the results for $\mu <1$.
	\end{proof}

By Lemma \ref{scaling}, we see that  a great advantage for the scaling in $E^s_\sigma$ is that $\varphi_\lambda \in E^s_\sigma$ vanishes as $\lambda \to \infty$ if the support set of $\widehat{\varphi}$ is away from the origin.

\subsection{Proof of Theorem \ref{1DNLHglobal}}

{\it Step 1.} First, we consider the case $s\leq 0$ and ${\rm supp} \, \widehat{u}_0 \subset \{\xi: |\xi|_\infty \geq 1\}$ and $\|u_0\|_{E^s_\sigma}$ is sufficiently small. Let us consider the mapping
\begin{align*}
\mathcal{T}:		u(t) \to  e^{t\Delta}u_0+\int_0^t e^{(t-\tau)\Delta}u^m(\tau)~d\tau
\end{align*}
in the space
\begin{align*}
		\mathcal{D}= \{u: \, {\rm supp }\, \widehat{u(t)} \subset \{\xi: |\xi|_\infty \geq 1\}, \   \|u\|_{X_{s,\sigma}}  \leq \delta\},
	\end{align*}
where
\begin{align*}
	\|u\|_{X_{s,\sigma}}	=  \|u\|_{\widetilde{L} ^m (0,\infty; E^{s,\sigma+2/m}_{2,2})  }
	\end{align*}
By Corollary \ref{linearestimateapp} and Lemma \ref{nonlinearestimate}, we have
\begin{align*}
\|\mathcal{T}u\|_{X_{s,\sigma}}   \leq C  \|u_0\|_{E^s_\sigma} + C \|u\|^m_{X_{s,\sigma}},
\end{align*}
and similarly,
\begin{align*}
\|\mathcal{T} u -  \mathcal{T} v\|_{X_{s,\sigma}}   \leq C (\|u\|^{m-1}_{X_{s,\sigma}} + \|v\|^{m-1}_{X_{s,\sigma}}) \|u- v\|_{X_{s,\sigma}}.
\end{align*}
Taking $\delta =2C\|u_0\|_{E^s_\sigma}$, we see that for sufficiently small $\|u_0\|_{E^s_\sigma}$  satisfying
\begin{align} \label{smallcod}
  C \delta^{m-1} \leq 1/100,
\end{align}
 we have $\mathcal{T}: \mathcal{D}\to \mathcal{D}$ is a contraction mapping and it has a unique fixed point $u\in X_{s,\sigma}$, i.e.,
\begin{align*}
 u(t) =  e^{t\Delta}u_0+\int_0^t e^{(t-\tau)\Delta}u^m(\tau)~d\tau
\end{align*}
in the space $X_{s,\sigma}$. Moreover, by Corollary \ref{linearestimateapp} and Lemma \ref{nonlinearestimate},
\begin{align} \label{Xsbdd}
\|u\|_{\widetilde{L} ^1 (0,\infty; E^{s,\sigma+2}_{2,2})  \cap \widetilde{L} ^\infty (0,\infty; E^{s,\sigma}_{2,2})} \leq 2 C  \|u_0\|_{E^s_\sigma}.
\end{align}
Furthermore, we can show that $u\in C([0,\infty); E^s_\sigma)$. In fact, By Corollary \ref{linearestimateapp}
\begin{align*}
 \left\|\int_{t_0}^t e^{(t-\tau)\Delta}u^m(\tau)~d\tau \right\|_{L^{\infty}(t_0,t_1;  E^s_\sigma)}
& \leq \left \|\int_{t_0}^t e^{(t-\tau)\Delta}u^m(\tau)~d\tau \right\|_{ \widetilde{L} ^\infty (t_0,t_1; E^{s,\sigma}_{2,2})} \\
  &  \lesssim  \|u^m \|_{\widetilde{L} ^1 (t_0,t_1; E^{s,\sigma}_{2,2})}.
 \end{align*}
By Lemma \ref{nonlinearestimate},
 \begin{align*}
   \|u^m \|_{\widetilde{L} ^1 (t_0,t_1; E^{s,\sigma}_{2,2})} \lesssim      \|u\|^m_{\widetilde{L} ^m (t_0,t_1; E^{s,\sigma+2/m}_{2,2})} \to 0, \ t_0\to t_1.
 \end{align*}
 It follows that $u\in C([0,\infty); E^s_\sigma)$.

Now, we consider the global solution in the case $s<0$ and ${\rm supp }\, \widehat{u}_0 \subset \mathbb{R}^d_I \setminus \{0\}$. Then there exists $\varepsilon_0 >0$ such that $ \{\xi: |\xi|_\infty \geq \varepsilon_0\}$.  It suffices to consider the case that $\|u_0\|_{E^s_\sigma}$ is large.  Denote
\begin{align} \label{u0lambda}
u_{0,\lambda}(x) = \lambda^{2/(m-1)} u_0(\lambda x).
\end{align}
In view of Lemma \ref{scaling}, we have
$$
\|u_{0,\lambda} \|_{E^s_\sigma} \leq \lambda^{2/(m-1) -d/2 + 0\vee \sigma} 2^{s(\lambda-1) \varepsilon_0} \|u_0\|_{E^s_\sigma}.
$$
Noticing that $s<0$, we see that
$$
\lim _{\lambda \to\infty }  \lambda^{2/(m-1) -d/2+ 0\vee \sigma} 2^{s(\lambda-1) \varepsilon_0 } =0.
$$
Hence, $\|u_{0,\lambda} \|_{E^s_\sigma}$  is sufficiently small if $\lambda\gg 1$. Moreover, $\widehat{u}_{0,\lambda} \subset \{\xi: |\xi|_\infty \geq \varepsilon_0 \lambda\} \subset \{\xi: |\xi|_\infty \geq 1\}$ if $\lambda>1/ \varepsilon_0 $.  Using the above result, we get a unique solution $u_\lambda$ of the integral equation
\begin{align*}
 v(t) =  e^{t\Delta}u_{0,\lambda}+\int_0^t e^{(t-\tau)\Delta}v(\tau)^m~d\tau
\end{align*}
in the space $X_{s,\sigma}$. Put
$$
u(t,x) = \lambda^{-2/(m-1)} u_\lambda (\lambda^{-2}t, \lambda^{-1}x).
$$
We need to consider the scaling property in $X_{s,\sigma}$.

\begin{lemma} \label{scalingXs}
		Assume $s\leq 0$, $f \in \widetilde{L} ^1 (0,\infty; E^{s,\sigma+2}_{2,2})  \cap \widetilde{L} ^\infty (0,\infty; E^{s,\sigma}_{2,2})$.  Write  $  f_{1/\lambda} (t,x) = f(t/\lambda^2, x/\lambda )$. Then we have
		\begin{align*}
			\| f_{1/\lambda}\|_{\widetilde{L} ^1 (0,\infty; E^{s\lambda,\sigma+2}_{2,2})  \cap \widetilde{L} ^\infty (0,\infty; E^{s\lambda,\sigma}_{2,2})} \leq 2^{C\lambda} \|f\|_{\widetilde{L} ^1 (0,\infty; E^{s,\sigma+2}_{2,2})  \cap \widetilde{L} ^\infty (0,\infty; E^{s,\sigma}_{2,2})}, \ \ \lambda >1.
		\end{align*}
\end{lemma}

	\begin{proof}[\textbf{Proof}]
By the definition of $X_{s,\sigma}$ and Lemma \ref{interlinearestimate},
we have
\begin{align*}
				\| f_{1/\lambda}\|_{\widetilde{L} ^1 (0,\infty; E^{s\lambda,\sigma+2}_{2,2})  \cap \widetilde{L} ^\infty (0,\infty; E^{s\lambda,\sigma}_{2,2})} \leq &  \left(\sum_{k \in \mathbb{Z}^d} \langle k\rangle^{2\sigma}  2^{2 s\lambda |k|}  \|\chi_{k+[0,1)^d} \widehat{f_{1/\lambda}}\|^2_{L^\infty_tL^2_x} \right)^{1/2} \\
 & +     \left(\sum_{k \in \mathbb{Z}^d} \langle k\rangle^{2(2+\sigma)}  2^{2 s\lambda |k|}  \|\chi_{k+[0,1)^d} \widehat{f_{1/\lambda}}\|^2_{L^1_tL^2_x} \right)^{1/2} .
\end{align*}
Since
\begin{align*}
\langle k\rangle^{\sigma} & 2^{  s\lambda |k|}  \|\chi_{k+[0,1)^d} \widehat{f_{1/\lambda}}\|_{L^\infty_tL^2_x}\\
 & \leq \lambda^{d/2 + 0\vee (-\sigma)}2^{|s|d\lambda} \| \langle\xi\rangle^{\sigma} 2^{s|\xi|}\chi_{\lambda k+[0,\lambda)^d} \widehat{f}\|_{L^\infty_tL^2_x} \\
& \leq \lambda^{d/2 + 0\vee (-\sigma)}2^{|s|d\lambda} \sum_{\ell \in  \mathbb{Z}^d \cap \lambda k+[0,\lambda)^d }\| \langle\xi\rangle^{\sigma} 2^{s|\xi|}\chi_{\ell +[0,1)^d} \widehat{f}\|_{L^\infty_tL^2_x}\\
& \leq \lambda^{d + 0\vee (-\sigma)}2^{|s|d\lambda} \left(\sum_{\ell \in  \mathbb{Z}^d \cap \lambda k+[0,\lambda)^d }\| \langle\xi\rangle^{\sigma} 2^{s|\xi|}\chi_{\ell +[0,1)^d} \widehat{f}\|^2_{L^\infty_tL^2_x}\right)^{1/2},
\end{align*}
one has that
\begin{align*}
 &  \left(\sum_{k \in \mathbb{Z}^d} \langle k\rangle^{2 \sigma }  2^{2 s\lambda |k|}  \|\chi_{k+[0,1)^d} \widehat{f_{1/\lambda}}\|^2_{L^\infty_tL^2_x} \right)^{1/2} \leq    2^{C\lambda} \|f\|_{\widetilde{L} ^\infty (0,\infty; E^{s,\sigma}_{2,2})}.
  \end{align*}
Similarly,
$$
\langle k\rangle^{2+ \sigma} 2^{  s\lambda |k| }  \|\chi_{k+[0,1)^d} \widehat{f_{1/\lambda}}\|_{L^1_tL^2_x} \leq \lambda^{2+d/2 + 0\vee (-2-\sigma)}2^{|s|d\lambda} \|\langle\xi\rangle^{2+\sigma} 2^{s|\xi|}\chi_{\lambda k+[0,\lambda)^d} \widehat{f}\|_{L^1_tL^2_x}.
$$
It follows that
\begin{align*}
 &  \left(\sum_{k \in \mathbb{Z}^d} \langle k\rangle^{2(2+\sigma) }  2^{2 s\lambda |k|}  \|\chi_{k+[0,1)^d} \widehat{f_{1/\lambda}}\|^2_{L^1_tL^2_x} \right)^{1/2} \leq    2^{C\lambda} \|f\|_{\widetilde{L}^1 (0,\infty; E^{s,\sigma}_{2,2})}.
  \end{align*}
We get the result, as desired.
	\end{proof}

By Lemmas \ref{scaling} and \ref{scalingXs}, $u \in X_{s\lambda,\sigma}$ and it satisfies
\begin{align*}
 u(t) =  e^{t\Delta}u_{0}+\int_0^t e^{(t-\tau)\Delta}u^m(\tau)~d\tau.
\end{align*}
Taking
\begin{align} \label{s0}
 s_0 =\lambda s ,
\end{align}
we have shown the global existence and uniqueness of solutions in (i) of Theorem \ref{1DNLHglobal}.

\begin{rem} \label{s0exp}
We can give an explanation to $s_0 =\lambda s$.  Recall that $\lambda$ can be chosen from \eqref{smallcod}. In order to \eqref{smallcod} holds, by Lemma \ref{scaling}, it suffices to take a $\lambda>1$ satisfying
$$
C(\lambda^{2/(m-1) -d/2 + 0\vee \sigma} 2^{s(\lambda -1) \varepsilon_0} \|u_0\|_{E^s_\sigma})^{m-1} \leq 1/100,
$$
where $C$ is a constant that only depends on the upper bound of the inequalities in Corollary \ref{linearestimateapp} and Lemma \ref{nonlinearestimate}. So, $\lambda=\lambda (\varepsilon_0, m,\sigma,s, \|u_0\|_{E^s_\sigma})$. For example, in the case $\sigma = d/2- 2/(m-1)$, we can take
$$
s_0  =s -( C +  \log_2^{1+ \|u_0\|_{E^s_\sigma}})   \varepsilon^{-1}_0   .
$$

\end{rem}

{\it Step 2.}  Now we consider the small data in $E^0_\sigma =H^\sigma $, we have shown the global existence and uniqueness of solutions in  $ C([0,\infty); E^{0}_\sigma)\cap \widetilde{L}^\infty (0,\infty; \, E^{0, \sigma}_{2,2} ) \cap \widetilde{L}^1 (0,\infty; \, E^{0, 2+ \sigma}_{2,2} )$ if initial data $u_0 \in H^\sigma$ small enough and
$$
{\rm supp} \, \widehat{u}_0 \subset \mathbb{R}^d_I \cap \{\xi: \, |\xi|_\infty \geq 1\}.
$$
Recall that for the initial data ${\rm supp} \, \widehat{u}_0 \subset \mathbb{R}^d_I \setminus \{0\}$, there exists $\varepsilon_0 >0$ such that
$$
{\rm supp} \, \widehat{u}_0 \subset \mathbb{R}^d_I \cap \{\xi: \, |\xi|_\infty \geq \varepsilon_0\}.
$$
Again, by the scaling argument, we see that $u_{0,\lambda}$ defined in \eqref{u0lambda} satisfying
$$
{\rm supp} \, \widehat{u}_{0, \lambda} \subset \mathbb{R}^d_I \cap \{\xi: \, |\xi|_\infty \geq \varepsilon_0 \lambda \}
$$
and
$$
\|{u}_{0, \lambda}\|_{H^\sigma} \leq C \lambda^{\sigma -d/2 +2/(m-1)} \|u_0\|_{H^\sigma}.
$$
Taking $\lambda =1/\varepsilon_0$ and
$$
 C \|u_0\|_{H^\sigma} \leq \delta \varepsilon_0^{\sigma -d/2 +2/(m-1)}.
$$
for some small $\delta >0 $ as in \eqref{smallcod}, we see that \eqref{1DNLHI} with initial data $u_{0,\lambda}$ has a unique global solution
 $u_\lambda \in  C([0,\infty); H^\sigma)\cap \widetilde{L}^\infty (0,\infty; \, E^{0, \sigma}_{2,2} ) \cap \widetilde{L}^1 (0,\infty; \, E^{0, 2+ \sigma}_{2,2} )$. Then
$u(t,x) = \lambda^{-2/(m-1)} u_\lambda (\lambda^{-2}t, \lambda^{-1}x).$  is the  unique global solution of  \eqref{1DNLHI} with initial data $u_{0}$
in the spaces $C([0,\infty); H^\sigma)\cap \widetilde{L}^\infty (0,\infty; \, E^{0, \sigma}_{2,2} ) \cap \widetilde{L}^1 (0,\infty; \, E^{0, 2+ \sigma}_{2,2} )$. So, we finish the proof of the part of existence and uniqueness of solutions in (ii) of Theorem \ref{1DNLHglobal}.

\vspace*{0.2cm}

{\it Step 3. }  We consider the ill-posedness issues of SLH
\begin{equation}\label{2illmodel}
		 \partial_t u - \Delta u  -      u^{m} =0, \  \ u (0,x) = u_0(x),
\end{equation}
Let $s<0$. Put
	\begin{equation}\label{mchooseinitial}
		\varphi   = \sum_{k\in 2^{4\mathbb{N}}, k>2d} \mathscr{F}^{-1}(2^{-sd|k|/2}(\chi_{[(m-1)k,(m-1)k+1/2 )^d}+\chi_{[-k-1/2(m-1),-k)^d})), \ \ u_0 = \delta \varphi,
	\end{equation}
we have $\|u_0\|_{E^s_\sigma}  \leq C$. Let $u_\delta$ be the solution of \eqref{2illmodel} with initial datum $ u_0=\delta\varphi$. Then
	\begin{align*}
\left.\frac{\partial u_\delta(t)}{\partial\delta }\right|_{\delta = 0} = e^{t\Delta} \varphi, ..., 	\	\left.\frac{\partial^m u_\delta(t)}{\partial \delta^m}\right|_{\delta = 0} = m!   \int_0^t e^{(t-\tau)\Delta}(e^{\tau\Delta}\varphi)^m d\tau.
	\end{align*}

	\begin{lemma} \label{Illposedness}
Let $s<0$. Suppose that $u_0$ is defined by \eqref{mchooseinitial}. Then, for any $s_0 \leq s$,
		\begin{align*}
			 \left\|\int_0^t e^{(t-\tau)\Delta}(e^{\tau\Delta} \varphi )^m d\tau \right\|_{E^{s_0}_\sigma} = \infty,\quad \forall ~t>0.
		\end{align*}
	\end{lemma}
	\begin{proof} Let us observe that for $\xi=\xi_1+...+\xi_m$, $e=(1,1,...,1)$,
\begin{align*}
  \widehat{e^{\tau\Delta}\varphi}*...*\widehat{e^{\tau\Delta}\varphi} (\xi)  & \geq     \int_{\mathbb{R}^d} e^{-\tau (|\xi_1|^2+...+|\xi_m|^2) } \widehat{\varphi}(\xi_1)...\widehat{\varphi}(\xi_m) d\xi_1...d\xi_{m-1} \\
   & \geq    2^{-s m d k /2} \int_{\mathbb{R}^d} e^{-\tau (|\xi_1|^2+...+|\xi_m|^2) }\prod^{m-1}_{j=1} \chi_{[-k-1/2(m-1),-k)^d} (\xi_j) \\
   & \quad\quad \quad\quad \quad\quad \times \chi_{[(m-1)k , (m-1)k+1/2)^d}  (\xi_m) d\xi_1...d\xi_{m-1} \\	
   & =   2^{-s m d k /2} \int_{\mathbb{R}^d} e^{-\tau (|\xi_1-ke|^2+...+|\xi_{m-1}-ke|^2 +|\xi_{m}+(m-1)ke|^2)  } \\
   & \quad\quad \quad\quad \quad\quad \times \prod^{m-1}_{j=1} \chi_{[-1/2(m-1),0)^d} (\xi_j) \chi_{[0 , 1/2)^d}  (\xi_m) d\xi_1...d\xi_{m-1} \\ 		&  \geq   2^{-s m d k /2} e^{-\tau C d^2 m^2 k^2 } \int_{\mathbb{R}^d}
    \prod^{m-1}_{j=1} \chi_{[-1/2(m-1),0)^d} (\xi_j) \chi_{[0 , 1/2)^d}  (\xi_m) d\xi_1...d\xi_{m-1} \\
    &  :=   2^{-s m d k /2} e^{-\tau C d^2 m^2 k^2 } \psi(\xi)
\end{align*}
for all $k\in 2^{4\mathbb{N}}$. Note that ${\rm supp} \, \psi \subset [-1/2,1/2)^d$.  So, we have
\begin{align*}
	  \mathscr{F}\int_0^t e^{(t-\tau)\Delta}(e^{\tau\Delta}\varphi )^m dt
& = \int_0^t e^{-(t-\tau)|\xi|^2}(\widehat{e^{\tau\Delta}\varphi}*...*\widehat{e^{\tau\Delta}\varphi})(\xi) d\tau \\
			& \gtrsim  2^{-s m d k /2}  e^{- t |\xi|^2} \psi \int_0^t  e^{-  \tau C d^2 m^2 k^2  } d\tau  \\
			&\gtrsim  2^{-sm dk/2}\frac{e^{-td}}{k^2} (1- e^{-tC d^2 m^2 k^2})  \psi .
		\end{align*}
Hence, we immediately obtain that
$$
 \left\|\int_0^t e^{(t-\tau)\Delta}(e^{\tau\Delta}u_0)^m dt\right\|_{E^{s_0}_{\sigma}}  =\infty
$$
by taking $k\to \infty$.  So, the mapping $u_0 \to u$ is not $C^m$ from $E^s_\sigma$ into $E^{s_0}_\sigma$ for any $s_0\leq s$.
	\end{proof}

We have shown a stronger result, i.e., the corresponding integral equation of \eqref{2illmodel} has no iteration solution in any $E^{s_0}_\sigma$ for the initial data as in \eqref{mchooseinitial}. In fact, let us observe the iteration sequence
$$
u^{(j+1)}(t)= \delta e^{t\Delta}\varphi +   \int_0^t e^{(t-\tau)\Delta}u^{(j)}(\tau)^m d\tau, \ \ u^{(0)} =0.
$$
It follows that
$$
\widehat{u^{(j+1)}(t)}= \delta e^{t|\xi|^2} \widehat{\varphi} +   \int_0^t e^{(t-\tau)|\xi|^2}\widehat{u^{(j)}(\tau)} *...*  \widehat{u^{(j)}(\tau)} d\tau, \ \ \widehat{u^{(0)}} =0.
$$	
Since $\widehat{\varphi}(\xi) \geq 0$ for all $\xi \in \mathbb{R}^d$, we see that
$$
\widehat{u^{(2)}(t)} \geq \delta e^{t|\xi|^2} \widehat{\varphi} +  \delta^m \int_0^t e^{(t-\tau)|\xi|^2} \widehat{e^{\tau\Delta}\varphi}*...*\widehat{e^{\tau\Delta}\varphi}   d\tau.
$$	
Hence we have
$$
\|u^{(2)}(t)\|_{E^{s_0}_\sigma} = \infty, \ \ \forall \ t>0.
$$
Since $\widehat{u^{(j)}}$  is nonnegative for any $j\geq 1$,  we immediately have $\|u^{(j)}(t)\|_{E^{s_0}_\sigma} = \infty$ for any $ t>0$. It follows that SLH \eqref{model} has no iteration solution in $E^{s_0}_\sigma$ for any $s_0\leq s$.

\vspace{0.2cm}

Now we consider the ill-posedness for the solutions in $H^\sigma$, $\sigma < d/2-2/(m-1)$. Put
	\begin{equation}\label{2chooseinitial}
		\psi (\xi)  = \prod^d_{j=1} \chi_{[1/2d,1/d) } (\xi(j)), \ \  \varphi= N^{-\sigma -d/2} \psi(\xi/N),   \ \ u_0 = \delta \varphi,
	\end{equation}
we have $\|\varphi\|_{H^\sigma} \sim 1$. Let $u_\delta$ be the solution of \eqref{2illmodel} with initial datum $ u_0=\delta\varphi$. Then
	\begin{align*}
\left.\frac{\partial u_\delta(t)}{\partial\delta }\right|_{\delta = 0} = e^{t\Delta} \varphi, \ 		\left.\frac{\partial^m u_\delta(t)}{\partial\delta^m}\right|_{\delta = 0} =  m!  \int_0^t e^{(t-\tau)\Delta}(e^{\tau\Delta}\varphi)^m d\tau.
	\end{align*}
 We have for $\xi_m =\xi-\xi_1-...-\xi_{m-1}$,
\begin{align*}
\left.\widehat{\frac{\partial^m u_\delta(t)}{\partial\delta^m}}\right|_{\delta = 0}	 &	=   	m!  \int_0^t e^{-(t-\tau)|\xi|^2}     (\widehat{e^{\tau\Delta} \varphi} *...* \widehat{e^{\tau\Delta} \varphi} )  d\tau \\
 &	=m! N^{-(\sigma+d/2)m} 	e^{-t |\xi|^2}  \int_0^t \int    e^{ \tau (|\xi|^2-\sum^m_{j=1} |\xi_j|^2)} \varphi \left(\frac{\xi_1}{N}\right)... \varphi \left(\frac{\xi_m}{N}\right) d\tau  d\xi_1...d\xi_{m-1} \\
 &	=m! N^{-(\sigma+d/2)m} 	e^{-t |\xi|^2}    \int  \frac{e^{ t (|\xi|^2-\sum^m_{j=1} |\xi_j|^2)}-1 }{|\xi|^2-\sum^m_{j=1} |\xi_j|^2}  \varphi \left(\frac{\xi_1}{N}\right)... \varphi \left(\frac{\xi_m}{N}\right)   d\xi_1...d\xi_{m-1}.
\end{align*}
Taking $t_N \sim 1/N^2$, one sees that
$$
 \frac{e^{ t_N (|\xi|^2-\sum^m_{j=1} |\xi_j|^2)}-1 }{|\xi|^2-\sum^m_{j=1} |\xi_j|^2} =  \frac{e^{ 2 t_N \sum_{i<j} \xi_i \xi_j}-1 }{ 2\sum_{i<j} \xi_i \xi_j }  \gtrsim \frac{1}{N^2}
$$
in the support set of $\varphi(\xi_1/N)... \varphi(\xi_m/N)$. So, we have
\begin{align*}
& \left\|\left. \frac{\partial^m u_\delta(t_N)}{\partial\delta^m}\right|_{\delta = 0}\right\|_{H^\sigma}\\ 	 &	 = \left\|\langle \xi\rangle^{\sigma} \left. \widehat{\frac{\partial^m u_\delta(t_N)}{\partial\delta^m}}\right|_{\delta = 0} \chi_{[mN/2, mN]}(|\xi|)\right\|_{2} \\
&	\gtrsim N^{-m(\sigma+d/2)-2+\sigma}  \left\| \int  \varphi \left(\frac{\xi_1}{N}\right)...  \varphi \left(\frac{\xi_{m-1}}{N}\right) \varphi \left(\frac{\xi-\xi_1-...-\xi_{m-1}}{N}\right)   d\xi_1...d\xi_{m-1} \right\|_{2}\\
&	\gtrsim N^{(m-1)(d/2-\sigma)-2 }.
\end{align*}
Noticing that $\sigma < d/2-2/(m-1)$, we see that $(m-1)(d/2-\sigma)-2 >0$. Taking $N\to \infty$, we immediately obtain that
\begin{align*}
& \left\|\left. \frac{\partial^m u_\delta(t_N)}{\partial\delta^m}\right|_{\delta = 0}\right\|_{H^\sigma} \to \infty.
\end{align*}
So, we have shown the ill-posedness part of (ii) in Theorem \ref{1DNLHglobal}.

\vspace{0.5cm}

{\it Step 4.} We show the error estimates for the iteration solutions. The iteration equations are
$$
u^{(j+1)}=   e^{t\Delta} u_0  +   \int_0^t e^{ (t-\tau)\Delta} (u^{(j)}(\tau))^m  d\tau, \ \ u^{(0)} =0.
$$	
Denote $\widehat{u^{(j)}} = v^{j}$, $\widehat{u}_0 = v_0$. For convenience, we denote $*^m a=\underset{m}{\underbrace{a*...*a}}$. It follows that $v^{j}$ satisfies
\begin{align} \label{iteseq}
 v^{j+1} =   e^{-t|\xi|^2} v_0  +   \int_0^t e^{-(t-\tau)|\xi|^2} (*^m v^{j}) (\tau)  d\tau, \ \ v^0 =0.
\end{align}	
By induction we have
\begin{lemma} \label{supportiterat}
Let $v^{j}$ be the iteration solution of \eqref{iteseq}. Then we have for all $j\geq 1$,
$$
{\rm supp}\, (v^{j+1} - v^{j}) \subset \{ \xi\in \mathbb{R}^d_I: \,   |\xi|\geq j (m-1) \varepsilon_0 \}.
$$
\end{lemma}
\begin{proof}
It is easy to see that
$$
v^1=  e^{-t|\xi|^2} v_0, \ \ v^2 =  e^{-t|\xi|^2} v_0  +   \int_0^t e^{-(t-\tau)|\xi|^2} (*^m v^1) (\tau)  d\tau.
$$
So, we have ${\rm supp}\, (v^{2} - v^1) \subset \{ \xi\in \mathbb{R}^d_I: \,   |\xi|\geq  (m-1) \varepsilon_0 \} $ and ${\rm supp}\,  v^{2} \subset \{\xi\in \mathbb{R}^d_I:\, |\xi| \geq \varepsilon_0\}$.
Let us observe that
\begin{align} \label{iteseq-v}
 v^{j+1} -v^{j} =    \int_0^t e^{-(t-\tau)|\xi|^2} \left((*^m v^{j}) - (*^{m} v^{j-1}) \right) (\tau)  d\tau.
\end{align}
Using the identity
$$
(*^m v^{j}) - (*^{m} v^{j-1}) = (v^{j}- v^{j-1}) * \left(\sum^{m-1}_{\ell=0} (*^{m-1-\ell} v^{j} ) * (*^{\ell} v^{j-1} ) \right),
$$
by induction we have the result, as desired.
\end{proof}

Using the fact that
$$
v^{j+r}  = v^{j} + \sum^r_{\ell=1} (v^{j+\ell} - v^{j+\ell -1})
$$
and Lemma \ref{supportiterat}, we have for any $r\in \mathbb{N}$,
$$
v^{j+r}(t,\xi) \chi_{\{\xi\in \mathbb{R}^d_I: |\xi|< (m-1) j \varepsilon_0\} }= v^{j}(t,\xi)  \chi_{\{\xi\in \mathbb{R}^d_I: |\xi|< (m-1) j \varepsilon_0\}  }.
$$
By (i) of Theorem \ref{1DNLHglobal}, we see that for any $t>0$,
$$
 \|u(t)-u^{(j)}(t) \|_{E^{s_0}_\sigma} \to 0, \ \ j\to \infty.
$$
It follows that
$$
v (t,\xi) \chi_{\{\xi\in \mathbb{R}^d_I:\,|\xi|< (m-1) j \varepsilon_0\} }= v^{j}(t,\xi)  \chi_{\{\xi\in \mathbb{R}^d_I:\,|\xi|<(m-1)j  \varepsilon_0 \} },
$$
which means that
\begin{align} \label{supportsetv}
{\rm supp}\, (v-v^{j}) \subset \left\{\xi\in \mathbb{R}^d_I: |\xi|\geq (m-1) j  \varepsilon_0 \right\}.
\end{align}

\begin{lemma} \label{Errorlem}
Let  $ 1\leq p, q\leq \infty$, $\sigma, \tilde{s} \in \mathbb{R}$,  $A\gg 1$. Suppose that ${\rm supp}\, \widehat{f} \subset \{\xi\in \mathbb{R}^d: |\xi|_\infty \geq A \}$. Then we have
$$
\left\|\int^t_0 e^{(t-\tau)\Delta} f(\tau)d\tau\right\|_{\widetilde{L}^\infty(\mathbb{R}_+, E^{\tilde{s},\sigma}_{p,q})} \lesssim \frac{1}{A^2}\|f\|_{\widetilde{L}^\infty(\mathbb{R}_+, E^{\tilde{s},\sigma}_{p,q})}.
$$
\end{lemma}
\begin{proof}
Let us recall the decay estimate (cf. \cite{WaZhGu2006})
$$
\|\Box_k e^{t\Delta} u_0\|_p \lesssim  e^{-c t|k|^2}\|\Box_k u_0\|_p,  \ k\in \mathbb{Z}^d_I, \ |k|_\infty \geq 1.
$$
It follows that
$$
\left\|\Box_k \int^t_0 e^{(t-\tau)\Delta} f(\tau)d\tau\right\|_p \lesssim \int^t_0 e^{-c(t-\tau)|k|^2} \|\Box_k f(\tau)\|_p d\tau.
$$
Using Young's inequality, we have for $|k|_\infty \geq A$
\begin{align}
\left\|\Box_k \int^t_0 e^{(t-\tau)\Delta} f(\tau)d\tau\right\|_{L^\infty_tL^p_x(\mathbb{R}_+\times \mathbb{R}^d) } & \lesssim \| e^{-ct|k|^2}\|_{L^1_t} \|\Box_k f(\tau)\|_{L^\infty_tL^p_x(\mathbb{R}_+\times \mathbb{R}^d) } \nonumber\\
& \lesssim \frac{1}{|k|^2} \|\Box_k f(\tau)\|_{L^\infty_tL^p_x(\mathbb{R}_+\times \mathbb{R}^d) } \nonumber\\
& \lesssim \frac{1}{A^2} \|\Box_k f(\tau)\|_{L^\infty_tL^p_x(\mathbb{R}_+\times \mathbb{R}^d) }. \label{errorest1}
\end{align}
\eqref{errorest1} is multiplied by $2^{\tilde{s}|k|} \langle k\rangle^\sigma$ and then equipped with the $\ell^q$-norm, we have the result, as desired.
\end{proof}

Using the integral equation,
\begin{align} \label{iteseq1}
v- v^{j} =     \int_0^t e^{-(t-\tau)|\xi|^2}(*^m v- *^m v^{j-1} )(\tau) d\tau,
\end{align}	
and taking $p=2,\, q=1, \sigma=0, \tilde{s}=\tilde{s}_0 <s_0$ in Lemma \ref{Errorlem},  we have
\begin{align} \label{iteseq2}
\sum_{l\in \mathbb{Z}^d} & 2^{\tilde{s}_0 |l|} \|v- v^{j}\|_{L^\infty_t L^2_\xi(\mathbb{R}_+ \times (l+[0,1)^d))}  \nonumber\\
& \lesssim \frac{1}{|j|^2} \sum_{l\in \mathbb{Z}^d}  2^{\tilde{s}_{0}|l|}  \left\|(v-  v^{j-1}) * \sum^{m-1}_{\ell =0}(*^{m-1-\ell }v)* (*^\ell v^{j-1}) \right\|_{L^\infty_t L^2_\xi(\mathbb{R}_+ \times (l+[0,1)^d))}.
\end{align}	
To estimate the right hand side of \eqref{iteseq2}, we need the following
\begin{lemma}\label{errornonlinearestimate}
Let $\tilde{s}_0 \leq 0$. Suppose that ${\rm supp}\, f_j \subset \mathbb{R}^d_I$. Then we have
$$
\sum_{l\in \mathbb{Z}^d}  2^{\tilde{s}_0|l|}  \|f_1*...*f_m\|_{L^\infty_t L^2_\xi(\mathbb{R}_+ \times (l+[0,1)^d))}   \lesssim  \prod^{m}_{j=1} \sum_{l\in \mathbb{Z}^d}  2^{\tilde{s}_0|l|}  \|f_j\|_{L^\infty_t L^2_\xi(\mathbb{R}_+ \times (l+[0,1)^d))} .
$$
\end{lemma}

\begin{proof}
By Young's and H\"older's inequality, we have
\begin{align*}
\sum_{l\in \mathbb{Z}^d}  & 2^{\tilde{s}_0|l|}  \|f_1*...*f_m\|_{L^\infty_t L^2_\xi(\mathbb{R}_+ \times (l+[0,1)^d))}  \\
&\leq\sum_{l\in \mathbb{Z}^d_I } 2^{\tilde{s}_0|l|} \sum_{l_1,...,l_m \in \mathbb{Z}^d_I  }\left\| \chi_{l+[0,1)^d} \left((\chi_{l_1+[0,1)^d} f_1)*  \cdots *( \chi_{l_m+[0,1)^d} f_m)\right) \right\|_{L_{t}^\infty L_\xi^2 }\\
&\leq\sum_{l\in \mathbb{Z}^d_I } 2^{\tilde{s}_0|l|} \sum_{l_1,...,l_m \in \mathbb{Z}^d_I  }\| (\chi_{l_1+[0,1)^d} f_1)*  \cdots *( \chi_{l_m+[0,1)^d} f_m) \|_{L_{t}^\infty L_\xi^2 } \chi_{\{|l-l_1-...-l_m|\leq m\}}\\ 			
&\lesssim   \sum_{l_1,...,l_m \in \mathbb{Z}^d_I  } 2^{\tilde{s}_0|l_1+...+l_m|}\| (\chi_{l_1+[0,1)^d} f_1)*  \cdots *( \chi_{l_m+[0,1)^d} f_m) \|_{L_{t}^\infty L_\xi^2 } \\ 			
&\leq  \sum_{l_1,...,l_m \in \mathbb{Z}^d_I  } \prod^{m-1}_{j=1} 2^{\tilde{s}_0|l_j| }\|  \chi_{l_j+[0,1)^d} f_j\|_{L_{t}^\infty L_\xi^1} 2^{\tilde{s}_0|l_m| } \| \chi_{l_m+[0,1)^d} f_m  \|_{L_{t}^\infty L_\xi^2 }\\
&\leq  \sum_{l_1,...,l_m \in \mathbb{Z}^d_I  } \prod^{m}_{j=1} 2^{\tilde{s}_0|l_j| }\|  \chi_{l_j+[0,1)^d} f_j\|_{L_{t}^\infty L_\xi^2}.			
\end{align*}
We have the result, as desired.
\end{proof}
Using the result of (i) of Theorem \ref{1DNLHglobal}, we see that $u\in  \widetilde{L}^\infty (0,\infty; \, E^{s_0, \sigma}_{2,2} )$ and
$$
\left\|u^{(j)}-u \right\|_{\widetilde{L}^\infty (0,\infty; \, E^{s_0, \sigma}_{2,2} )} \to 0, \ \ k\to \infty.
$$
For any $\tilde{s}_0<s_0$,  by H\"older's inequality, we have
$$
\|u\|_{\widetilde{L}^\infty (0,\infty; \, E^{\tilde{s}_0,0}_{2,1} )} \leq \|\{2^{(\tilde{s}_0 -s_0)|k|} \langle k\rangle^{-\sigma}\}\|_{\ell^2} \|u\|_{\widetilde{L}^\infty (0,\infty; \, E^{{s}_0,\sigma}_{2,2} )} \lesssim  \|u\|_{\widetilde{L}^\infty (0,\infty; \, E^{ {s}_0,\sigma}_{2,2} )}
$$
So, by Lemma \ref{errornonlinearestimate} we obtain that
\begin{align*}
 \sum_{l\in \mathbb{Z}^d} 2^{\tilde{s}_{0}|l|}  & \left\|(v-  v^{j-1})  \sum^{m-1}_{\ell =0}(*^{m-1-\ell }v)* (*^\ell  v^{j-1}) \right\|_{L^\infty_t L^2_\xi(\mathbb{R}_+ \times (l+[0,1)^d))} \\
\lesssim &  \left( \sum_{l\in \mathbb{Z}^d}  2^{\tilde{s}_0|l|}  \left( \|v^{j-1}\|_{L^\infty_t L^2_\xi(\mathbb{R}_+ \times (l+[0,1)^d))}+ \|v\|_{L^\infty_t L^2_\xi(\mathbb{R}_+ \times (l+[0,1)^d))} \right) \right)^{m-1}\\
& \times  \sum_{l\in \mathbb{Z}^d_I }   2^{\tilde{s}_0|l|} \|v- v^{j-1}\|_{L^\infty_t L^2_\xi(\mathbb{R}_+ \times (l+[0,1)^d))} \\
\lesssim &  \left( \|u^{(j)}\|_{\widetilde{L}^\infty (0,\infty; \, E^{s_0, \sigma}_{2,2} )} +  \|u\|_{\widetilde{L}^\infty (0,\infty; \, E^{s_0, \sigma}_{2,2} )} \right)^{m-1}\\
& \times  \sum_{l\in \mathbb{Z}^d_I }   2^{\tilde{s}_0|l|} \|v- v^{j-1}\|_{L^\infty_t L^2_\xi(\mathbb{R}_+ \times (l+[0,1)^d))}
\end{align*}
Combining the above estimate with \eqref{iteseq2}, we have
\begin{align} \label{iteseq3}
\sum_{l\in \mathbb{Z}^d} & 2^{\tilde{s}_0|l|} \|v- v^{j}\|_{L^\infty_t L^2_\xi(\mathbb{R}_+ \times (l+[0,1)^d))}   \nonumber\\
&  \leq  \frac{C}{|j|^2} \sum_{l\in \mathbb{Z}^d}  2^{\tilde{s}_0|l|}   \|(v-  v^{j-1}) \|_{L^\infty_t L^2_\xi(\mathbb{R}_+ \times (l+[0,1)^d))}
\end{align}
By induction, we have
\begin{align} \label{iteseq3}
\|u^{(j)}(t)-u(t)\|_{E^{\tilde{s}_0}} \leq  \sum_{l\in \mathbb{Z}^d} & 2^{\tilde{s}_0|l|} \|v- v^{j}\|_{L^\infty_t L^2_\xi(\mathbb{R}_+ \times (l+[0,1)^d))}
  \leq  \frac{C^j}{( j !)^2}.
\end{align}
This proves (iii) of Theorem \ref{1DNLHglobal}.

\section{SLH for $m\geq 1+4/d$} \label{sectNLHhD}

In the case $m\geq 1+4/d$, we can improve the condition ${\rm supp}\, \widehat{u}_0 \subset   \{\xi\in \mathbb{R}^d_I: \,   |\xi(j)|_\infty \geq \varepsilon_0\}$ for some $\varepsilon_0>0$ by a weak condition ${\rm supp}\, \widehat{u}_0 \subset \mathbb{R}^d_I$.  For example, let $d=1$ and
\begin{align}
    u_0 &= A\,   \frac{d}{dx}\left(\delta (x) + \frac{2 \, \mathrm{ i}}{x}\right),  \  A\in \mathbb{C},  \label{example3}
\end{align}
we have $\widehat{u}_0 =  i \xi  \chi_{\{\xi: \,\xi \geq 0\}}$ and ${\rm supp} \widehat{u}_0 \subset [0,\infty)$. Comparing \eqref{example3} with \eqref{example1}, we see that the imaginary part of $u_0$ in \eqref{example1} changes  signs for infinite times and in \eqref{example3} never changes signs.

\subsection{Scaling in the critical cases}

Recalling that for the initial data $u_0 \in E^s_\sigma$ with ${\rm supp} \, \widehat{u}_0 \subset \{\xi:\, |\xi| \geq \varepsilon_0\}$, we have shown that $\|u_{0,\lambda}\|_{E^s_\sigma}$ vanishes as $\lambda\to \infty$ by using the exponential decay of $2^{s\lambda \varepsilon_0}$. However, in the critical case $\sigma=d/2-2/(m-1)$, it still holds $\lim_{\lambda\to \infty}\|u_{0,\lambda}\|_{E^s_\sigma}=0$ if ${\rm supp} \, \widehat{u}_0 \subset \mathbb{R}^d_I$.

\begin{lemma} \label{scalingcrit}
Let $s<0$, $\sigma \geq 0$. Denote $f_\lambda = \lambda^a f(\lambda\, \cdot)$. Assume that $\sigma+a = d/2$. Then for any $f\in E^s_\sigma$, we have
\begin{align}
\lim_{\lambda\to +\infty} \|f_\lambda\|_{E^s_\sigma} =0. \label{scalingcrit1}
\end{align}
\end{lemma}
\begin{proof}
Noticing that $\sigma+a =d/2$, We have for any $\lambda>1$
\begin{align}
  \|f_\lambda\|_{E^s_\sigma}   = \lambda^{-d+a} \|\langle\xi\rangle^{\sigma} 2^{s|\xi|} \widehat{f}(\xi/\lambda)\|_2
     \leq    \|\langle \xi\rangle^{\sigma} 2^{s\lambda|\xi|} \widehat{f} \|_2.
   \label{scalingcrit2}
\end{align}
Since $f\in E^s_\sigma$, we have
$$
\lim_{\delta\to 0}   \|\chi_{\{|\xi| \leq \delta \}} \langle \xi\rangle^{\sigma} 2^{s |\xi|} \widehat{f} \|_2 =0.
$$
Hence, for any $\varepsilon>0$, there exists $\delta_0 >0$ such that
\begin{align}
    \|\chi_{\{|\xi| \leq \delta_0 \}} \langle \xi\rangle^{\sigma} 2^{s\lambda|\xi|} \widehat{f} \|_2 \leq  \|\chi_{\{|\xi| \leq \delta_0 \}} \langle \xi\rangle^{\sigma} 2^{s |\xi|} \widehat{f} \|_2 <\varepsilon/2.
   \label{scalingcrit3}
\end{align}
On the other hand,
\begin{align*}
    \|\chi_{\{|\xi| > \delta_0 \}} \langle \xi\rangle^{\sigma} 2^{s\lambda|\xi|} \widehat{f} \|_2 \leq 2^{s(\lambda-1)\delta_0} \|\chi_{\{|\xi| \leq \delta_0 \}} \langle \xi\rangle^{\sigma} 2^{s |\xi|} \widehat{f} \|_2  \leq 2^{s(\lambda-1)\delta_0} \|f\|_{E^s_\sigma}.
\end{align*}
Choosing $\lambda_0 $ sufficiently large such that
$$
 2^{s(\lambda-1)\delta_0} \|f\|_{E^s_\sigma} < \varepsilon/2,
$$
we have for any $\lambda >\lambda_0$,
\begin{align}
    \|\chi_{\{|\xi| > \delta_0 \}} \langle \xi\rangle^{\sigma} 2^{s\lambda|\xi|} \widehat{f} \|_2 < \varepsilon/2.
   \label{scalingcrit4}
\end{align}
In view of \eqref{scalingcrit2}--\eqref{scalingcrit4}, we have for $\lambda>\lambda_0$,
\begin{align}
\|f_\lambda\|_{E^s_\sigma}  \leq \|\chi_{\{|\xi| \leq \delta_0 \}} \langle \xi\rangle^{\sigma} 2^{s\lambda|\xi|} \widehat{f} \|_2 +  \|\chi_{\{|\xi| > \delta_0 \}} \langle \xi\rangle^{\sigma} 2^{s\lambda|\xi|} \widehat{f} \|_2 < \varepsilon,
   \label{scalingcrit5}
\end{align}
which implies the result.
\end{proof}

\subsection{Result in the case $m\geq 1+4/d$}

For convenience, we denote
\begin{align}
\|u\|_{X^{s,\sigma}} = & \|u\|_{ \widetilde{L}^\infty (0,\infty; E^{s,\sigma}_{2,2}(\mathbb{Z}^d_I\setminus \{0\})) \cap \widetilde{L}^1(0,\infty; E^{s,\sigma+2}_{2,2}(\mathbb{Z}^d_I\setminus \{0\}) ) }  \nonumber \\
& + \|\Box_0 u\|_{L^\infty_t L^2_x \cap L^m_t L^{2m}_x (\mathbb{R}_+\times \mathbb{R}^d) }, \label{criticalspace}
\end{align}

\begin{thm} \label{hDNLHglobal}
		Let $d\geq 1$, $f(u)=u^m$, $m\in  \mathbb{N}$,  $m\geq 1+4/d$, $\sigma=d/2-2/(m-1)$. Let $s<0$, $u_0\in E^{s}_{\sigma}$ with $\mathrm{supp}~\widehat{u}_0\subset \mathbb{R}^d_{I} $. Then there exists $s_0:= s_0(m, u_0) \leq s$ such that SLH \eqref{model} has a unique solution
$u\in X^{s_0,\sigma}$ satisfying the following equivalent integral equation:
\begin{align} \label{hDNLHI}
u(t) = e^{t\Delta} u_0 + \int_0^t e^{(t-\tau)\Delta}u(\tau)^m~d\tau.
\end{align}
Moreover, $u \in C([0,\infty); E^{s_0}_{\sigma}) $.
\end{thm}

Theorem \ref{hDNLHglobal} needs several remarks.
\begin{itemize}

\item[(i)] Let $u_{0,\lambda}$ be as in \eqref{u0lambda}.  From the proof of Lemma \ref{scalingcrit}, we see that $\lim_{\lambda\to \infty}\|u_{0,\lambda}\|_{E^s_\sigma}=0$ and the vanishing process of $\|u_{0,\lambda}\|_{E^s_\sigma}$ depends on both $u_0$ and its norm. It follows that $s_0 \leq s$ also depends on $u_0$ and its norm in $E^s_\sigma$. However, we can take $s=s_0$ if $u_0\in E^{s}_{\sigma}$  with $\mathrm{supp}~\widehat{u}_0\subset \mathbb{R}^d_{I}$  is sufficiently small.

\item[(iii)] By the same reason as in Theorem \ref{1DNLHglobal}, condition ${\rm supp}\, \widehat{u}_0 \subset \mathbb{R}^d_I$ is necessary for the existence of the global solutions.

\end{itemize}

\subsection{Proof of Theorem \ref{hDNLHglobal} }

{\it Step 1.} We prove the result for sufficiently small initial data $u_0 \in E^s_\sigma$.   We need a lemma to control the lower frequency part of the solution, see \cite{WaHuHaGu2011}.

\begin{lemma} \label{lowerfreqcont}
Let $2\leq p <\infty$, $2/\gamma \leq d(1/2-1/p)$. Then we have
\begin{align}
\|\Box_0 e^{t\Delta} u_0\|_{L^\gamma_t L^p_x (\mathbb{R}_+\times \mathbb{R}^d)} & \lesssim \|\Box_0 u_0\|_2, \label{lowerfreq1} \\
\left\|\Box_0 \int^t_0 e^{(t-\tau)\Delta} f(\tau) d\tau \right\|_{L^\gamma_t L^p_x (\mathbb{R}_+\times \mathbb{R}^d)} &  \lesssim \left\|\Box_0 f \right\|_{L^1_t L^2_x (\mathbb{R}_+\times \mathbb{R}^d)}
\end{align}
\end{lemma}
Let us consider the mapping
\begin{align*}
\mathcal{T}: u(t) \to  e^{t\Delta}u_0+\int_0^t e^{(t-\tau)\Delta}u^m(\tau)~d\tau
\end{align*}
in the space
\begin{align*}
		\mathcal{D}= \{u: \, {\rm supp }\, \widehat{u(t)} \subset \mathbb{R}^d_I, \   \|u\|_{X^{s,\sigma}}  \leq \delta\},
	\end{align*}
where $\|u\|_{X^{s,\sigma}}$  is defined in \eqref{criticalspace}.  By Lemma \ref{lowerfreqcont},
\begin{align}
\|\Box_0 \mathcal{T} u \|_{L^\infty_t L^{2}_x \cap L^m_t L^{2m}_x }
\leq \,  &  \, \|\Box_0 e^{t\Delta} u_0\|_{ L^\infty_t L^{2}_x \cap L^m_t L^{2m}_x  } \nonumber\\
 & \,  +  \left\|\Box_0 \int^t_0 e^{(t-\tau)\Delta} u(\tau)^m  d\tau \right\|_{L^\infty_t L^{2}_x \cap L^m_t L^{2m}_x } \nonumber\\		
 \lesssim   \,  & \|  u_0\|_{2}     +  \left\|\Box_0 (u^m )\right\|_{L^1_t L^{2}_x }. 	
 \end{align} 	
Using $\sum_k \Box_k =I$, we have
\begin{align}
  \left\|\Box_0 (u^m )\right\|_{L^1_t L^{2}_x} \leq \sum_{k_1,...,k_m \in \mathbb{Z}^d_I} \|\Box_0(\Box_{k_1}u \Box_{k_2}u...\Box_{k_m} u) \|_{L^1_t L^{2}_x}.	
 \end{align}
Since ${\rm supp} \,  \widehat{u(t)}  \subset \mathbb{R}^d_I$, we see that
$$
\Box_0(\Box_{k_1}u \Box_{k_2}u...\Box_{k_m} u) =0
$$
if there exists $k_j$ such that $|k_j|_\infty \geq 1$. Hence
\begin{align}
  \left\|\Box_0 (u^m )\right\|_{L^1_t L^{2}_x} =   \|\Box_0(\Box_{0}u  )^m  \|_{L^1_t L^{2}_x} \leq \| (\Box_{0}u  )^m  \|_{L^1_t L^{2}_x} \leq \| \Box_{0}u   \|^m_{L^m_t L^{2m}_x} .	
 \end{align}
It follows that
\begin{align} \label{criticalspace21}
\|\Box_0 \mathcal{T} u \|_{L^\infty_t L^{2}_x \cap L^m_t L^{2m}_x }		
& \lesssim     \| \Box_0 u_0\|_{2}     +    \| \Box_{0}u   \|^m_{L^m_t L^{2m}_x }   \nonumber\\
 &  \lesssim     \|  u_0\|_{E^s_\sigma}     +    \| u \|^m_{X^{s,\sigma}}. 	
 \end{align}
Next, by Corollary \ref{linearestimateapp}, we can estimate
\begin{align}
  \|\mathcal{T}  u &\|_{ \widetilde{L}^\infty (0,\infty; E^{s,\sigma}_{2,2}(\mathbb{Z}^d_I\setminus \{0\})) \cap \widetilde{L}^1(0,\infty; E^{s,\sigma+2}_{2,2}(\mathbb{Z}^d_I\setminus \{0\}) ) }  \nonumber \\
 \leq \  &  \|e^{t\Delta} u_0\|_{  \widetilde{L}^\infty (0,\infty; E^{s,\sigma}_{2,2}(\mathbb{Z}^d_I\setminus \{0\})) \cap \widetilde{L}^1(0,\infty; E^{s,\sigma+2}_{2,2}(\mathbb{Z}^d_I\setminus \{0\}) ) } \nonumber\\
 & \,  +  \left\| \int^t_0 e^{(t-\tau)\Delta} u(\tau)^m  d\tau \right\|_{ \widetilde{L}^\infty (0,\infty; E^{s,\sigma}_{2,2}(\mathbb{Z}^d_I\setminus \{0\})) \cap \widetilde{L}^1(0,\infty; E^{s,\sigma+2}_{2,2}(\mathbb{Z}^d_I\setminus \{0\}) ) } \nonumber\\
\lesssim  &  \    \|  u_0\|_{E^s_\sigma}  +  \|u^m \|_{\widetilde{L}^1(0,\infty; E^{s,\sigma}_{2,2}(\mathbb{Z}^d_I\setminus \{0\}) )}.
 \label{criticalspace10}
\end{align}
One has that
\begin{align}
  \|u^m \| _{\widetilde{L}^1(0,\infty; E^{s,\sigma}_{2,2}(\mathbb{Z}^d_I\setminus \{0\}) )}
& \leq  \left(\sum_{k\in \mathbb{Z}^d\setminus \{0\}} 2^{2s|k|} \langle k\rangle^{2\sigma} \|\Box_k (\Box_0 u)^m\|_{L^1_tL^2_x} \right)^{1/2} \nonumber\\
& \quad\quad +  \left(\sum_{k\in \mathbb{Z}^d\setminus \{0\}} 2^{2s|k|} \langle k\rangle^{2\sigma} \|\Box_k (u^m-(\Box_0 u)^m)\|_{L^1_tL^2_x} \right)^{1/2}\nonumber\\
& : =I+II.
 \label{criticalspace11}
\end{align}
Noticing that $\Box_k (\Box_0 u)^m \neq 0$  implies that $|k|_\infty \leq m$, we have
\begin{align}
  I &  \lesssim  m^{d/2+  \sigma }  \| \Box_0 u \|^m_{L^m_tL^{2m}_x}.  \label{criticalspace12}
\end{align}
Applying Remark \ref{appendnonlinearest},
\begin{align}
II  \leq C_m \|u  \|^{m-1}_{\widetilde{L}^\infty(0,\infty; E^{s,\sigma}_{2,2}  )}  \|u  \| _{\widetilde{L}^1(0,\infty; E^{s,\sigma+2}_{2,2}(\mathbb{Z}^d_I\setminus \{0\}) )}.
 \label{criticalspace15}
\end{align}
It is easy to see that
\begin{align}
 \|u  \|_{\widetilde{L}^\infty(0,\infty; E^{s,\sigma}_{2,2}  )}  \leq  \|\Box_0 u\|_{L^\infty_tL^2_x} +  \|u  \|_{\widetilde{L}^\infty(0,\infty; E^{s,\sigma}_{2,2}(\mathbb{Z}^d_I \setminus \{0\})  )}\leq \|u\|_{X^{s,\sigma}}.
 \label{criticalspace16}
\end{align}
By \eqref{criticalspace15} and \eqref{criticalspace16},
\begin{align}
II  \leq C_m \|u\|^m_{X^{s,\sigma}}.
 \label{criticalspace18}
\end{align}	
Then, by \eqref{criticalspace10},  \eqref{criticalspace11}, \eqref{criticalspace12} and \eqref{criticalspace18},
\begin{align}
  \|\mathcal{T}  u  \|_{ \widetilde{L}^\infty (0,\infty; E^{s,\sigma}_{2,2}(\mathbb{Z}^d_I\setminus \{0\})) \cap \widetilde{L}^1(0,\infty; E^{s,\sigma+2}_{2,2}(\mathbb{Z}^d_I\setminus \{0\}) ) }
\lesssim  &  \    \|  u_0\|_{E^s_\sigma}  +  C_m \|u \|^m_{X^{s,\sigma} }.
 \label{criticalspace19}
\end{align}
By \eqref{criticalspace21} and \eqref{criticalspace19}
\begin{align}
  \|\mathcal{T}  u  \|_{X^{s,\sigma}}
\lesssim  &  \    \|  u_0\|_{E^s_\sigma}  + C_m \|u \|^m_{X^{s,\sigma} }.
 \label{criticalspace20}
\end{align}
So, it follows from \eqref{criticalspace20} that SLH \eqref{model} has a unique solution in $X^{s,\sigma}$  if $\|  u_0\|_{E^s_\sigma}$ is sufficiently small.

{\it Step 2.} We consider the large initial data in $E^s_\sigma$.  The idea is the same as in the proof of Theorem \ref{1DNLHglobal}. By Lemma \ref{scalingcrit}, $u_{0,\lambda} = \lambda^{2/(m-1)} u(\lambda\, \cdot)$ satisfies
$$
\lim_{\lambda\to \infty} \|u_{0,\lambda}\|_{E^s_\sigma} =0.
$$
So, for any $u_0\in E^s_\sigma$, we can choose a $\lambda \gg 1$ such that
$$
  \|u_{0,\lambda}\|_{E^s_\sigma}  \ll 1.
$$
It follows that
\begin{align} \label{hDNLHIlambda}
v(t) = e^{t\Delta} u_{0,\lambda} + \int_0^t e^{(t-\tau)\Delta}v(\tau)^m~d\tau.
\end{align}
has a unique solution $u_\lambda \in X^{s,\sigma}$. Taking $u= \lambda^{-2/(m-1)} u(t/\lambda^2, x/\lambda)$, we can repeat the procedures as in the proof of Theorem \ref{1DNLHglobal} to get the result of Theorem \ref{hDNLHglobal}.

\section{SLH with an exponential nonlinearity} \label{sectNLHexp}

The heat equation with  exponential nonlinearities has been studied in \cite{Fu2018,GaVa1997,LaTz1993,Va1999},  where the global existence and blowup behavior for the radial solutions were considered in the real-valued cases.
We will study the Cauchy problem for the heat equation with an exponential nonlinearity
	\begin{equation}\label{NLHexp}
		 u_t-\Delta u = e^u-1 , \ \  u(0,x) = u_0(x),
	\end{equation}
where $u(t,x)$ is a complex-valued function of $(t,x) \in [0,\infty)\times \mathbb{R}^d$, $u_0 \in {E}^s_{\sigma}$ with $s<0$ and   $\widehat{u}_0$ is supported in the first octant and away from the unit ball, say
\begin{align} \label{supportexp}
{\rm supp} \, \widehat{u}_0 \subset \{\xi \in  \mathbb{R}^d : \ \xi_i \geq 0, \ i=1,...,d, \ |\xi|_\infty \geq 2\}.
\end{align}
The main result of this section is the following

\begin{thm} \label{NLHexpr}
Let $s<0$, $\sigma \geq d/2$,  and $u_0 \in E^s_{\sigma}$ satisfy \eqref{supportexp} Then there exists $s_0 \leq s$ such that \eqref{NLHexp} has a unique solution
$$
u\in C(\mathbb{R}_+, E^{s_0}_{\sigma}) \cap \widetilde{L} ^\infty (0,\infty; E^{s_0,\sigma}_{2,2}) \cap  \widetilde{L} ^1 (0,\infty; E^{s_0,\sigma+2}_{2,2})
$$
in the sense that $u$ satisfies the equivalent integral equation
 \begin{align} \label{NLHexpi}
 u(t) = e^{t(I+\Delta)} u_0 + \int^t_0  e^{(t-\tau)(I+\Delta)} (e^{u(\tau)}-u(\tau)-1) d\tau.
\end{align}
\end{thm}
Recalling that
  \begin{align} \label{NLHexp1}
  e^{u}   = \sum^\infty_{m=0} \frac{u^m}{m!},
\end{align}
we see that Fujita's critical and subcritical powers are contained in the Taylor expansion of $e^{u}$ in the cases $m=2,3$ for $d=1$, and $m=2$ for $d=2$. Moreover, the linear part $u$ contained in the Taylor's expansion of $e^u$ plays a bad role which prevents the global existence of solutions. Indeed, let us observe the semi-group $e^{t(I+\Delta)}$, its lower frequency part has no time-decay and  one easily sees that
$$
e^{t/2} \|P_{\leq 1/2} u_0\|_2 \leq  \|e^{t(I+\Delta)} P_{\leq 1/2}u_0\|_2  \leq e^t \|u_0\|_2, \ \ P_{\leq 1/2} = \mathscr{F}^{-1} \chi_{|\xi|\leq 1/2} \mathscr{F}.
$$
So, the lower frequency part of the free solutions tends to infinity as $t\to \infty$.   This is why we assume that the support set of the Fourier transform of initial data is away from  the unit ball in \eqref{supportexp}.

\subsection{Linear Estimates}

Now, let us make a scaling to \eqref{NLHexp}. Denote
$$
u_\lambda (t,x) = u(\lambda^2 t, \lambda x).
$$
If $u$ is a solution of \eqref{NLHexp}, we easily see that $u_\lambda$ solves
	\begin{equation}\label{NLHexpscaling}
		 u_t-\Delta u = \lambda^2 (e^u-1)  , \ \  u(0,x) = u_0(\lambda x).
	\end{equation}
\eqref{NLHexpscaling} can be rewritten as
	\begin{equation*}
		 u_t-(\lambda^2+ \Delta) u - \lambda^2 (e^u-u-1) =0 , \ \  u(0,x) = u_0(\lambda x),
	\end{equation*}
which is essentially equivalent to the following integral equation
 \begin{align} \label{NLHexpiscaling}
 u(t) = e^{t(\lambda^2 I+\Delta)} u_0(\lambda\, \cdot) +  \lambda^2 \int^t_0  e^{(t-\tau)(\lambda^2 I+\Delta)} (e^{u(\tau)}-u(\tau)-1) d\tau.
\end{align}
We can first solve \eqref{NLHexpiscaling}. For the sake of convenience, we write
$$
H_\lambda (t) : =  e^{t(\lambda^2 I+\Delta)} = \mathscr{F}^{-1} e^{-t(|\xi|^2 -\lambda^2  )} \mathscr{F}.
$$
Let us consider the estimate of $H_\lambda(t)$. As indicated in the above, $H_\lambda(t)$ has no time-decay for $|\xi| \leq \lambda$.

\begin{lemma} \label{Lpmultiplier}
Let $1\leq p \leq \infty$, $\lambda \gg  1$. Assume that $k\in \mathbb{Z}^d$ and $|k|_\infty \geq 2\lambda$. Then we have
\begin{align}
\label{estimatep}
\|\Box_k H_\lambda (t) u_0 \|_p \leq C e^{-ct|k|^2} \|\Box_k u_0\|_p,
\end{align}
where $C$ is independent of $\lambda$ and $k$.
\end{lemma}
{\bf Proof}.  Let $\sigma$ be a smooth cut-off function verifying $\sigma(\xi)=1$ for $\xi \in [0,1]^d$ and $\sigma(\xi) =0$ if $\xi \not\in [-1/4, 5/4]^d$. Put $\sigma_k = \sigma (\cdot-\xi)$.  By Young's inequality,
$$
\|\Box_k H_\lambda (t) u_0 \|_p \leq  \left\|\mathscr{F}^{-1} \left(\sigma_k e^{t(\lambda^2-|\xi|^2)}\right ) \right\|_1  \|\Box_k u_0\|_p.
$$
Using the translation invariance,
$$
  \left\|\mathscr{F}^{-1} \left(\sigma_k e^{t(\lambda^2-|\xi|^2)}\right ) \right\|_1  = e^{t(\lambda^2-|k|^2)} \left\|\mathscr{F}^{-1} \left(\sigma  e^{-t( |\xi|^2 + 2k\xi) }\right ) \right\|_1.
$$
Noticing that ${\rm supp }\,\sigma$ is compact, in view of Proposition \ref{Nikolskii} we have
$$
   \left\|\mathscr{F}^{-1} \left(\sigma  e^{-t( |\xi|^2 + 2k\xi) }\right ) \right\|_1 \lesssim e^{Ct|k|}.
$$
Noticing that $|k|_\infty \geq  2\lambda$, one has the result as desired. $\hfill\Box$

\begin{lemma} \label{TSestimate}
Let $1\leq \gamma, \gamma_1, p \leq \infty$, $\gamma_1\leq \gamma$,  $\lambda \gg  1$. Assume that
\begin{align} \label{RIlambda}
{\rm supp}\, \widehat{u}_0, {\rm supp}\, \widehat{f(t)} \subset \mathbb{R}^{d}_{I,\lambda}:= \{\xi\in \mathbb{R}^d_I: |\xi|_\infty \geq 2\lambda\}.
\end{align}
 Then we have
\begin{align}
\|\Box_k H_\lambda (t) u_0 \|_{L^\gamma_t L^p_x } & \leq  C |k|^{-2/\gamma} \|\Box_k u_0\|_p,    \label{TSestimate1}\\
\left\|\int^t_0 \Box_k H_\lambda (t-\tau) f(\tau)d\tau \right\|_{L^\gamma_t L^p_x } & \leq C |k|^{-2/\gamma -2/\gamma'_1} \left\|  \Box_k   f \right\|_{L^{\gamma'_1}_t L^p_x },  \label{TSestimate2}
\end{align}
where $C$ is independent of $\lambda$ and $k$.
\end{lemma}
{\bf Proof}. Taking the $L^\gamma$ norm in both sides of \eqref{estimatep}, we immediately have \eqref{TSestimate1}.  By \eqref{estimatep},
\begin{align}
\left\|\int^t_0 \Box_k H_\lambda (t-\tau) f(\tau)d\tau \right\|_{p} \leq   \int^t_0  e^{-c(t-\tau)|k|^2} \|\Box_k f(\tau)\|_p d\tau.  \label{TSestimate3}
\end{align}
Taking  $L^\gamma$ norm in both sides of \eqref{TSestimate3} and using Young's inequality, we obtain \eqref{TSestimate2}. $\hfill\Box$

\begin{cor}\label{linearestimate}
Let $s\leq 0$, $\sigma \in \mathbb{R}$, $1\leq \gamma_1\leq \gamma \leq \infty$.  Suppose that $\mathrm{supp}~ \mathscr{F}_x f, \, {\rm supp}\, \widehat{u}_0\subset \mathbb{R}^d_{I,\lambda}$. Then there exists $C>0$ which is independent of $\lambda\gg 1$ such that
		\begin{align}
			\|H_\lambda (t) u_0\|_{ \widetilde{L} ^\gamma (0,\infty; E^{s,\sigma+2/\gamma}_{2,2})  } & \leq C \|u_0\|_{E^s_\sigma},  \label{expbasicest1} \\
			\left\|\int_0^t H_\lambda (t-\tau) f(\tau)~d\tau\right\|_{\widetilde{L} ^\gamma (0,\infty; E^{s,\sigma+2/\gamma}_{2,2}) } &\leq C \|f\|_{\widetilde{L} ^{\gamma_1} (0,\infty; E^{s,\sigma-2/\gamma'_1}_{2,2}) }.  \label{expbasicest2}
		\end{align}
	\end{cor}
\begin{proof}
Taking $p=2$, multiplying $2^{s|k|}\langle k\rangle^{\sigma+2/\gamma}$  and then equipping $\ell^2$ norms in both sides of  \eqref{TSestimate1} and   \eqref{TSestimate2}, we have the results, as desired.
\end{proof}

\subsection{Proof of Theorem \ref{NLHexpr} }

 Let us consider the scaling of initial data $u_{0, \lambda}= u_0(\lambda\, \cdot)$. By Lemma \ref{scaling} we see that
$$
\|u_{0, \lambda} \|_{E^s_\sigma} \leq C 2^{ s (\lambda-1)} \lambda^{\sigma -d/2} \|u_0\|_{E^s_\sigma} \to 0, \ \ \lambda\to +\infty.
$$
Let $\lambda \in \mathbb{N}$ be a sufficiently large number. Similar to  the above argument, Taking $\delta= 2C 2^{ s (\lambda-1)} \lambda^{\sigma -d/2} \|u_0\|_{E^s_\sigma} $ and
$$
X_{s, \sigma, \lambda} =\left\{u: \, \|u\|_{X_{s, \sigma, \lambda}} =\|u\|_{\widetilde{L} ^\infty (\mathbb{R}_+; E^{s,\sigma}_{2,2}) \cap  \widetilde{L} ^1 (\mathbb{R}_+; E^{s,\sigma+2}_{2,2})} \leq \delta, \ {\rm supp}\, \widehat{u(t)} \subset  \mathbb{R}^d_{I,\lambda} \right\},
$$
where $\mathbb{R}^d_{I,\lambda}$ is as in \eqref{RIlambda}. Let us consider the mapping
 \begin{align*}
\mathscr{T}_\lambda : u(t) \to e^{t( \lambda^2 I+\Delta)} u_{0,\lambda}  + \lambda^2 \int^t_0  e^{(t-\tau)(\lambda^2 I+\Delta)} (e^{u(\tau)}-u(\tau)-1) d\tau.
\end{align*}
Applying Corollary \ref{linearestimateapp}, Lemmas \ref{interlinearestimate} and \ref{nonlinearestimate}, we can get for any $u\in X_{s,\lambda}$, $\sigma\geq d/2$,
 \begin{align*}
\|\mathscr{T}_\lambda u \|_{X_{s,\sigma,\lambda}} &
\leq C \|u_0\|_{E^s_{\sigma}}  +  \lambda^2 C\|e^{u}-u -1 \|_{\widetilde{L} ^1 (\mathbb{R}_+; E^{s,\sigma}_{2,2}) } \\
&  \leq  C \|u_{0,\lambda}\|_{E^s_{\sigma}}  +  C\lambda^2 \sum^\infty_{m=2} \frac{1}{m!} \|u^m\|_{\widetilde{L} ^1 (\mathbb{R}_+; E^{s,\sigma}_{2,2}) }  \\
&  \leq  C \|u_{0,\lambda}\|_{E^s_{\sigma}}  +  C\lambda^2 \sum^\infty_{m=2} \frac{ C^m m^{m/2} }{m!} \|u\|^m_{\widetilde{L} ^m (\mathbb{R}_+; E^{s,\sigma+2/m}_{2,2}) }  \\
 &  \leq  C \|u_{0,\lambda}\|_{E^s_{\sigma}}  +  C \lambda^2 \sum^\infty_{m=2} \frac{  C^m m^{m/2}}{m!} \|u\|^m_{X_{s,\sigma,\lambda}}\\
 &  \leq  \delta/2  + C \lambda^2 \delta^2 \sum^\infty_{m=2} \frac{  C^m  m^{m/2}}{m!} \delta^{m-2}.
\end{align*}
If $\lambda $ is large enough, then we have $\delta \leq 1$ and
 \begin{align*}
    \sum^\infty_{m=2} \frac{ C^m  m^{m/2} }{m!} \delta^{m-2} \leq C.
\end{align*}
Moreover, we can assume that
$$
C^2 \lambda^2 \delta     \leq \frac{1}{2}.
$$
Hence, we have
 \begin{align*}
\|\mathscr{T}_\lambda u \|_{X_{s,\sigma,\lambda}}
     \leq \delta.
\end{align*}
Similarly
 \begin{align*}
\|\mathscr{T}_\lambda u - \mathscr{T}_\lambda v \|_{X_{s,\sigma, \lambda}}
    \leq  \frac{1}{2}  \| u - v \|_{X_{s,\sigma, \lambda}}.
\end{align*}
Hence, $\mathscr{T}_\lambda: X_{s,\sigma,\lambda} \to X_{s,\sigma, \lambda}$ has a fixed point $u_\lambda \in X_{s,\sigma, \lambda}$, which solves
 \begin{align*}
  u(t) = e^{t( \lambda^2 I+\Delta)} u_{0,\lambda}  + \lambda^2 \int^t_0  e^{(t-\tau)(\lambda^2 I+\Delta)} (e^{u(\tau)}-u(\tau)-1) d\tau.
\end{align*}
Taking
$$
u(t,x) = u_\lambda \left(\frac{t}{\lambda^2}, \frac{x}{\lambda}\right),
$$
by the scaling argument we see that $u$ satisfies
 \begin{align*}
  u(t) = e^{t(  I+\Delta)} u_{0 }  +   \int^t_0  e^{(t-\tau)( I+\Delta)} (e^{u(\tau)}-u(\tau)-1) d\tau.
\end{align*}
Following Lemma \ref{scalingXs} we have 	
$$
u\in \widetilde{L} ^\infty (0,\infty; E^{s_0, \sigma}_{2,2}) \cap  \widetilde{L} ^1 (0,\infty; E^{s_0,\sigma+2}_{2,2}), \ \ s_0=\lambda s.
$$
$u\in C(\mathbb{R}_+,  E^{s_0}_{\sigma})$ can be obtained by following the same way as in Theorem \ref{1DNLHglobal} and the details are omitted.

\section{Initial data in $E^s_{2,1}$} \label{SLHmodspaces}

Denote $Q_k : = k+[0,1)^d$ for $k \in \mathbb{Z}^d$, Let $s<0$. Define (cf. \cite{FeGrLiWa2021})
$$
E^{s}_{2,1}:=\left\{f\in \mathscr{S}_1':  \|f\|_{E^{s}_{2,1}} = \sum_{k\in \mathbb{Z}^d} 2^{s|k|} \| \widehat{f} \|_{L^2(Q_k) }<\infty \right\}.
$$
By the inclusion $l^1\subset l^2$, we see that $E^s_{2,1} \subset E^s_{\sigma}$ for any $\sigma \leq 0$.  By \eqref{equivnormes} and Cauchy-Schwarz inequality,  $E^{s}_{\sigma} \subset E^s_{2,1}$ for any $\sigma >d/2$. It follows that
\begin{align} \label{embeddingmod}
E^{s}_{\sigma_1} \subset E^s_{2,1} \subset E^s_{\sigma_2}, \ \ \forall \ \sigma_1 > d/2, \ \sigma_2 \leq 0
\end{align}
and the above embeddings  are 	sharp in the sense that
\begin{align} \label{embeddingmod2}
E^{s}_{d/2} \not\subset E^s_{2,1} \not\subset E^s_{\varepsilon}, \ \ \forall \  \varepsilon> 0.
\end{align}

$E^s_{2,1}$ has algebraic structures if it is localized in the first octant of frequency spaces. So, we can use a similar way as in the cases of initial data in $E^s_\sigma$ to obtain the following

\begin{thm} \label{hDNLHglobalmod}
		Let $d\geq 1$, $f(u)=u^m$,  $m\in  \mathbb{N}\setminus \{1\}$. Let $s<0$, $u_0\in E^{s}_{2,1}$ with $\mathrm{supp}~\widehat{u}_0\subset \mathbb{R}^d_{I}\setminus \{0\}$. Then there exists $s_0 \leq s$ such that SLH \eqref{model} has a unique solution
$$
u\in C([0,\infty); E^{s_0}_{2,1}) \cap \widetilde{L} ^\infty (0,\infty; E^{s_0}_{2,1})  \cap \widetilde{L}^1(0,\infty; E^{s_0}_{2,1})
$$
in the sense that $u$ satisfies the integral equation \eqref{1DNLHI}.
\end{thm}
For the Cauchy problem of the semi-linear heat equation with an exponential nonlinearity, we have
	
\begin{thm} \label{NLHexprmod}
Let $s<0$ and $u_0 \in E^s_{2,1}$ satisfy \eqref{supportexp} Then there exists $s_0 \leq s$ such that \eqref{NLHexp} has a unique solution
$$
u\in C(\mathbb{R}_+, E^{s_0}_{2,1}) \cap \widetilde{L} ^\infty (0,\infty; E^{s_0}_{2,1}) \cap  \widetilde{L} ^1 (0,\infty; E^{s_0}_{2,1})
$$
in the sense that $u$ satisfies the integral equation \eqref{NLHexpi}.
\end{thm}

By \eqref{embeddingmod} and \eqref{embeddingmod2}, we see that Theorems \ref{hDNLHglobalmod} and \ref{NLHexprmod} contain a class of  initial data not covered by the spaces $E^s_\sigma$ as in Theorems \ref{1DNLHglobal}, \ref{hDNLHglobal} and \ref{NLHexpr}. Now we sketch their proofs.

	\begin{lemma} \label{hDscalingmod}
		Assume $s < 0$, $f \in E^s_{2,1}$ and $\mathrm{supp}~\widehat{f}\subset   \mathbb{R}^d_{I,+} $, $\lambda \gg 1$.  Then
		\begin{align*}
			\|f_\lambda\|_{E^s_{2,1}}\leq C  2^{s(\lambda-1)}\|f\|_{E^s_{2,1}}.
		\end{align*}
	\end{lemma}
\begin{proof}
It suffices to consider the case $\mathrm{supp}~\widehat{f}\subset \{\xi:    \xi_1  \geq 1 \}\cap \mathbb{R}^d_{I,+}$. 	By the definition of $E^s_{2,1}$ and the support property of $\widehat{f}$, we have from Cauchy-Schwarz inequality that
\begin{align*}
\|f_\lambda\|_{E^s_{2,1}} &= \sum_{k}2^{s|k|} \lambda^{-d}\|\widehat{f}(\lambda^{-1}\xi)\|_{L^2_{\xi\in k+[0,1]^d}}\\
&= \sum_k\lambda^{-d/2}2^{s|k |}\|\widehat{f}(\xi)\|_{L^2_{\xi\in \lambda^{-1}(k+[0,1]^d)}}\\
&\leq \sum_{l:l_1\geq 1} \lambda^{-d/2} \sum_{k:\lambda^{-1}(k+[0,1]^d) \cap l+[0,1]^d\not=\emptyset}2^{s|k|}\|\widehat{f}(\xi)\|_{L^2_{\xi\in \lambda^{-1}(k+[0,1]^d)}}\\
& \lesssim \sum_{l:l_1\geq 1}   2^{s \lambda |l| } \| \widehat{f}(\xi) \|_{L^2_{\xi\in l+[0,1]^d}} \\
&\lesssim 2^{s (\lambda-1)} \|f\|_{E^s_{2,1}}.
\end{align*}
We immediately have the result, as desired.
	\end{proof}

\begin{lemma}\label{highlinearestimatemod}
		Assume $\mathrm{supp}\,\mathscr{F}_x f(t,\cdot), \mathrm{supp}~\widehat{u}_0\subset  \mathbb{R}^{d}_{I,+}$ for all $t> 0$. Then,
		\begin{align*}
			\|e^{t\Delta}u_0\|_{\widetilde{L} ^1 (0,\infty; E^{s}_{2,1}) \cap \widetilde{L} ^\infty (0,\infty; E^{s}_{2,1})  }&\lesssim \|u_0\|_{E^s_{2,1}},\\
			\left\|\int_0^t e^{(t-\tau)\Delta}f(\tau)~d \tau\right\|_{\widetilde{L} ^1 (0,\infty; E^{s}_{2,1}) \cap \widetilde{L} ^\infty (0,\infty; E^{s}_{2,1})   }&\lesssim \|f\|_{ \widetilde{L} ^1 (0,\infty; E^{s}_{2,1})  }.
		\end{align*}
	\end{lemma}

\begin{lemma} \label{2TSnonlinestimatemod}
Let $s<0$.
Assume that $ {\rm supp}\, \widehat{u}  \subset \mathbb{R}^d_I$. Then we have
\begin{align}
 \|u^m \|_{ \widetilde{L} ^1 (0,\infty; E^{s}_{2,1})  }
   \leq   C^m m^d  \|u\|^{m-1}_{\widetilde{L} ^\infty (0,\infty; E^{s}_{2,1})} \|u\|_{\widetilde{L} ^1 (0,\infty; E^{s}_{2,1})}
    \label{2TSnonlinestimate2mod}
    \end{align}
\end{lemma}
Lemmas \ref{highlinearestimatemod}  and  \ref{2TSnonlinestimatemod} can be obtained by following similar ways as in Section \ref{sectNLH1D} and we omit the details of their proofs. Then we can imitate the procedures as in Sections \ref{sectNLH1D} and \ref{sectNLHexp} to get the proofs of Theorems \ref{hDNLHglobalmod} and \ref{NLHexprmod}, respectively.
	
\section{An example of solutions for SLH } \label{sectNLHnusim}

Let $u $ be the solution of
\begin{equation}\label{simmodeldelta}
		 \partial_t u - \Delta u -   u^{2} =0, \  \ u (0,x) = \delta u_0(x).
\end{equation}
Let $d=1$, $v=\widehat{u}$, and $\widehat{u}_0 (\xi) = v_0 (\xi)$. Then $v$ satisfies
$$
v (\xi,t)=  \delta e^{-t|\xi|^2} v_0  +   \int_0^t e^{-(t-\tau)|\xi|^2}v  (\tau) *  v (\tau) d\tau
$$
Let  $\mathbb{R}^d_{I,+}$ be as in \eqref{rdplus}.  Suppose that ${\rm supp}\, v_0 \subset \mathbb{R}^d_{I,+}$.
 Using the support set of $v_0$ and similar to  \eqref{supportsetv}, one sees that for any $K\in \mathbb{N}$,
\begin{equation}\label{exampiterationmodelg}
v (t, \xi)= \sum^K_{k=1} \frac{\delta^k}{k!} \, \frac{\partial^k v}{\partial \delta^k}\Big|_{\delta=0}(t,\xi), \ \ |\xi| < K,   	
\end{equation}
where
 \begin{align}\label{exampiterationmodelb}
 \frac{\partial v}{\partial \delta }\Big|_{\delta=0} (t,\xi) & = e^{- t\xi^2} v_0(\xi),  \nonumber \\
  \frac{\partial^k v}{\partial \delta^k} \Big|_{\delta=0} (t,\xi) & =   \int_0^t e^{-(t-\tau)\xi^2 }\sum^{k-1}_{i=1} \binom{k}{i} \left( \frac{\partial^{k-i} v}{\partial \delta^{k-i}} \Big|_{\delta=0} \right)* \left(  \frac{\partial^{i} u}{\partial \delta^{i}}\Big|_{\delta=0} \right) (\tau,\xi) d\tau	.
\end{align}
\begin{example} \rm
Let $  H(\xi-1)$ be the translation of Heaviside function,
$$
H(\xi)=
\left\{
\begin{array} {ll}
1, &  \xi \geq 0,\\
0,  & \xi <0.
\end{array}
\right.
$$
 Let $ v_0 (\xi) = e^\xi H(\xi-1)$.
We have
\begin{align*}
\frac{\partial v}{\partial \delta }\Big|_{\delta=0} (t_1,\xi_1) & = e^{-t_1 \xi^2_1 +\xi_1} H(\xi_1-1), \\
\frac{\partial^2 v}{\partial \delta^2 }\Big|_{\delta=0} (t_2,\xi_2)
& =  2 e^{-t_2 \xi^2_2 +\xi_2 } H(\xi_2-2) \int^{t_2}_0 \int^{\xi_2 -1}_1  e^{2t_2\xi_1 (\xi_2-\xi_1)}  d\xi_1 dt_1 ,\\
\frac{\partial^3 v}{\partial \delta^3 }\Big|_{\delta=0} (t_3,\xi_3)
& = 2\cdot 3! \ e^{-t_3 \xi^2_3 + \xi_3} H(\xi_3-3) \int^{t_3}_0 \int^{\xi_3 -1}_2  e^{2t_2\xi_2 (\xi_3-\xi_2)}  \\ & \ \ \ \ \ \ \ \ \ \ \times \int^{t_2}_0 \int^{\xi_2 -1}_1  e^{2t_1\xi_1 (\xi_2-\xi_1)}  d\xi_1 dt_1  d\xi_2 dt_2.
\end{align*}
We see that $(\frac{\partial v}{\partial \delta } + \frac{1}{2}\frac{\partial^2 v}{\partial \delta^2 } + \frac{1}{6}\frac{\partial^3 v}{\partial \delta^3 }) \Big|_{\delta=0}(\xi,t)$ is the exact solution of
$$
v (\xi,t)=   e^{-t|\xi|^2} v_0  +   \int_0^t e^{-(t-\tau)|\xi|^2}v  (\tau) *  v (\tau) d\tau
$$
in the region $ \{\xi: \xi \leq 3 \}. $

\end{example}

\textbf{Claims.} On behalf of all authors, the corresponding author states that there is no conflict of interest.

\textbf{Acknowledgments.}  The second named author is very grateful to Professor Y. Giga for his enlightening discussions and for his valuable suggestions on the paper. Also, he would
 like to thank the reviewers for their valuable suggestions, which enable us to give an essential improvement to Theorem \ref{1DNLHglobal}.  The authors are supported in part by the NSFC, grants 11771024, 12171007.

	\bibliographystyle{amsplain}

	\scriptsize\textsc{Jie Chen and Baoxiang Wang: School of Sciences, Jimei University, Xiamen, 361021 and  School of Mathematical Sciences, Peking University,   Beijing 100871, P.~R.~of China}.
	
	\textit{E-mail address}: \verb"jiechern@pku.edu.cn" \ \ and  \ \  \verb"wbx@math.pku.edu.cn"
	\vspace{10pt}

\scriptsize\textsc{Zimeng Wang:  School of Mechanical and Aerospace Engineering, Queen University, Belfast,
Northern Ireland,
BT9 5AH}.
	
	\textit{E-mail address}: \verb"zwang70@qub.ac.uk"

\end{document}